%% file: main.tex
\documentclass[12pt]{article}
\usepackage[margin=1in]{geometry}
\usepackage{multirow}
\usepackage{graphicx,amsmath,amsthm,comment,framed,algorithm,algorithmic,subcaption,url,bm}
\title{A Strict Complementarity Approach to Error Bound and Sensitivity of Solution of Conic Programs}
\author{Lijun Ding\footnote{Wisconsin Institute for Discovery,, University of Wisconsin - Madison,
		Madison, WI 53715, USA;
		\texttt{lijunbrianding@gmail.edu}}  \hspace{0.8pt}  and Madeleine Udell\footnote{Management Science and Engineering, Stanford University, Stanford, CA 94305, USA;
		\texttt{udell@stanford.edu}    }
}
\input defs
\newcommand{\EE}{\mathbf{E}}
\newcommand{\FF}{\mathbf{F}}
\usepackage{xr}

\begin{document}
	
	\maketitle
	\begin{abstract}
		In this paper, we provide an elementary, geometric, and unified framework
		to analyze conic programs
		that we call the strict complementarity approach.
		This framework allows us to
		establish error bounds and quantify the sensitivity of the solution.
		The framework uses three classical ideas from convex geometry and
		linear algebra:
		linear regularity of convex sets,
		facial reduction,
		and orthogonal decomposition.
		We show how to use this framework to derive error bounds
		for linear programming (LP), second order cone programming (SOCP),
		and semidefinite programming (SDP).
	\end{abstract}
	
	\section{Introduction}
	Given two finite dimensional Euclidean spaces $\EE$ and $\FF$,
	each equipped with an inner product denoted as $\inprod{\cdot}{\cdot}$,
	we consider a conic program in standard form with decision variable $x\in \EE$:
	\begin{equation}\tag{$\mathcal{P}$}
		\begin{aligned}\label{opt: primalcone}
			& {\text{minimize}}
			& & \inprod{\cost}{x} \\
			& \text{subject to} & &  \Amap x = \bvector,\\
			& & & x\in \cone.
		\end{aligned}
	\end{equation}
	Here the problem data comprises a linear map $\Amap:\EE \rightarrow \FF$,
	a right hand side vector $\bvector\in \FF$, and a cost vector $c\in \EE$.
	The cone $\cone \subset \EE$ is proper \footnote{A cone $\cone$ is proper if it is closed, convex, and pointed.}.
	The solution set and optimal value of \eqref{opt: primalcone} are denoted as $\xset$ and $\pval$ respectively. When $\cone$ is the nonnegative orthant $\reals_+^\dm$, the second order cone $\SOC{\dm}$,
	or the set of positive semidefinite matrices $\sym_+^\dm$,
	we call the corresponding problem \eqref{opt: primalcone} an LP, SOCP, or SDP, respectively.
	
	In this paper, we provide an elementary framework based on \emph{strict complementarity} (see Section \ref{sec: assumptions})
	to establish \emph{error bounds} and quantify the \emph{sensitivity of the solution} of Problem \eqref{opt: primalcone}.
	In the following, $\twonorm{\cdot}$ denotes
	the Euclidean norm induced by the inner product, while $\norm{\cdot}$ is a
	generic norm that will be specified when we instantiate these bounds later
	in the paper.
	\begin{itemize}
		\item \textbf{Error bound:} Given $x\in \EE$, define three error metrics:
		suboptimality $\optgap(x)\defn \inprod{c}{x}-\pval$,
		linear infeasibility $\Amap x -b$,
		and conic infeasibility as $\conefeas{x}=x -\conpart{x}$,
		where the conic part $\conpart{x}:\,= \proj{\cone}(x)$.\footnote{
			The orthogonal projector $\proj{\cone}$ is defined as
			$ \proj {\cone} (x)= \argmin_{y\in \cone}\twonorm{y-x}$.
		}
		These errors are easy to measure, while the distance of a given point $x$ to the solution is not.
		This paper shows how to establish an error bound for \eqref{opt: primalcone}
		that bounds the distance to the solution in terms of these measurable error metrics,
		for some constants $c_i,i=1,2,3$ and exponent $p>0$	independent of $x$:
		\begin{equation}\tag{ERB}
			\begin{aligned}\label{eq: errorbound}
				\dist(x,\xset)^p\leq c_1 |\optgap(x)| + c_2 \norm{\Amap x-b} +c_3 \norm{\conefeas{x}},
			\end{aligned}
		\end{equation}
		where $	\dist(x,\xset):= \inf_{\xsol \in \xset}\twonorm{x-\xsol}$ is the distance to $\xsol$.
		
		\item \textbf{Sensitivity of the solution:}
		We often wish to understand how the solution of the problem changes with perturbations
		of the problem data.
		Given new problem data $(c',\Amap',b')\in \EE\times \mathbf{L}(\EE,\FF)\times\FF$ for
		Problem \eqref{opt: primalcone},
		where $\mathbf{L}(\EE,\FF)$ is the set of linear maps from $\EE$ to $\FF$,
		Problem \eqref{opt: primalcone} admits a new optimal solution set $\xset'$.
		This paper shows how to quantify the sensitivity of the solution,
		for some constants $c'_i,i=1,2,3$ and exponent $p'>0$, via the inequality
		\begin{equation}\tag{SSB}  \label{eq: sensitivity}
			\textbf{dist}^{p'}(\xset,\xset')\leq  c_1'\norm{c-c'}+c_2'\norm{\Amap - \Amap'}+c_3'\norm{b-b'},
		\end{equation}
		where $ \textbf{dist}(\xset,\xset'):\,= \inf_{\xsol\in \xset,x'\in \xset'}\twonorm{\xsol-\xsol'}$
		is the distance between solution sets.
		In fact, once an error bound of the form \eqref{eq: errorbound} is available,
		we can prove an inequality of this form
		by bounding the error metrics of the new solution $\xsol'\in \xset'$
		with respect to the original problem data $(c, \Amap, b)$ 
		in terms of
		the perturbation $(c-c',\Amap - \Amap',b-b')$.
	\end{itemize}
	
	\paragraph{Importance of the error bound and sensitivity of solution.}
	The error bound and sensitivity of the solution
	can be regarded as condition numbers for Problem \eqref{opt: primalcone}.
	They guarantee that the output of iterative algorithms to solve \eqref{opt: primalcone}
	is still useful despite optimization error (of the algorithm) and
	measurement error (of the problem data) \cite{lewis2014nonsmooth,ding2020regularity}.
	The error bound is also vital in proving faster convergence for first order algorithms \cite{drusvyatskiy2018error,zhou2017unified,necoara2019linear,johnstone2017faster}.
	Hence a huge body of work has devoted to establish error bounds and sensitivity of solutions
	~\cite{hoffman1952approximate,drusvyatskiy2018error,zhou2017unified,lewis2014nonsmooth,sturm2000error,nayakkankuppam1999conditioning}.
	
	\paragraph{Our Contribution.} In this paper,
	we use the notion of \emph{strict complementarity} (defined in Section \ref{sec: assumptions})
	to provide an elementary, geometric, and unified framework,
	described in detail in Section \ref{sec: The strict complementary slackness approach},
	to establish bounds of the form
	\eqref{eq: errorbound} and \eqref{eq: sensitivity}
	for the conic program \eqref{opt: primalcone}. 
	Specifically, in Section \ref{sec: conicexamples} and \ref{sec: fromerrorboundtosensitivityofsolution},
	we show how to construct a bound with exponents
	$p=p'=1$ for LP and $p=p'=2$ for SOCP and SDP, under strict complementarity,
	and provide a way to obtain explicit estimates of $c_i,i=1,2,3$ in terms of
	the primal and dual solutions and problem data when the primal solution is unique.
	Table \ref{tb: PAndconstantAndBoundf} summarizes our results.
	
	The main contribution of this paper is a new and simple framework for proving bounds of this form.
	As discussed in Section \ref{sec: discussion},
	many particular bounds that we present here have been discovered before.
	On the other hand, we believe that some of the bounds are new:
	in particular, bounds on the sensitivity of the solution
	that pertain when the primal or dual solution are not unique.
	
	\paragraph{Paper organization.} The rest of the paper is organized as follows.
	In Section \ref{sec: assumptions},
	we discuss two important analytical conditions assumed throughout this paper:
	strong duality and dual strict complementarity.
	In Section \ref{sec: Defining The strict complementary slackness approach},
	we describe the basic framework of the strict complementarity approach:
	linear regularity of convex sets, facial reduction, and extension via orthogonal decomposition.
	In Section \ref{sec: conicexamples},
	we apply the framework to specific examples, LP, SOCP, and SDP, to establish error bounds.
	We next demonstrate how to use the error bound established
	to characterize the sensitivity of solutions
	in Section \ref{sec: fromerrorboundtosensitivityofsolution} by
	bounding the error measures of the new solution $\xsol'\in \xset'$ in terms of the
	perturbation $(c-c',\Amap - \Amap',b-b')$.
	Finally, we discuss previous results regarding \eqref{eq: errorbound} and \eqref{eq: sensitivity},
	how this work relates to them, and potential new directions.
	
	\paragraph{Notation.} We use $\EE,\FF,\EE',\FF'$ to represent generic finite dimensional Euclidean spaces.
	For a set $\mathcal{C}$ in $\EE$, we denote its interior, boundary, affine hull, and relative interior
	as $\intr(\mathcal{C})$, $\partial \mathcal{C}$, $\affine{\mathcal{C}}$, and $\relint(\mathcal{C})$ respectively.
	We equip $\reals^\dm$ with the dot inner product,
	and $\sym^\dm$ and $\reals^{\dm\times \dm}$ with the trace inner product.
	The distance to a set $\mathcal{C}$ is defined as $\dist(x,\mathcal{C})=\inf_{z\in \mathcal{C}}\twonorm{x-z}$.
	We write $\norm{\cdot}$ for an arbitrary norm, and $\twonorm{\cdot}$ for the $\ell_2$ norm induced by the underlying inner product.
	For matrices, the operator norm (maximum singular value), Frobenius norm, and nuclear norm (sum of singular values) are
	denoted as $\opnorm{\cdot}$, $\fronorm{\cdot}$, and $\nucnorm{\cdot}$ respectively.
	For a linear map $\mathcal{B}:\mathbf{E}\rightarrow \mathbf{F}$ and a linear space $\Vspace\subset\mathbf{E}$,
	we write the restriction of
	$\mathcal{B}$ to $\Vspace$ as $\mathcal{B}_{\Vspace}$.
	We define the largest and smallest singular value of $\mathcal{B}$ as
	$\sigma_{\max } (\mathcal{B}):\,=\max_{\twonorm{x}=1}\twonorm{\mathcal{B}(x)}$ and $\sigma_{\min} (\mathcal{B}):\,=\min_{\twonorm{x}=1}\twonorm{\mathcal{B}(x)}$ respectively.
	
	\section{The strict complementary slackness approach}\label{sec: The strict complementary slackness approach}
	In Section \ref{sec: assumptions},
	we introduce two important structural conditions, strong duality
	and dual strict complementarity, that are essential to our approach.
	Next in Section \ref{sec: Defining The strict complementary slackness approach},
	we describe the main ingredients of the strict complementary slackness approach:
	linear regularity (Section \ref{sec: linearRegularity}),
	facial reduction (Section \ref{sec: FacialReduction}),
	and orthogonal decomposition (Section \ref{sec: extension}).
	Our main result, Theorem \ref{thm: basicframework}, is in Section \ref{sec: extension}.
	
	\subsection{Analytical conditions} \label{sec: assumptions}
	Here we define two conditions that are essential to our framework: strong duality and
	dual strict complementarity.
	To start, let us recall the dual problem of \eqref{opt: primalcone} is
	\beq \label{opt: d}\tag{$\mathcal{D}$}
	\ba{ll}
	\mbox  {maximize} & \inprod{\bvector}{y} \\
	\mbox{subject to}  & \cost- \Amap^* y \in \cone^*. \\
	\ea
	\eeq
	The vector $y\in \FF$ is the decision variable,
	the linear map $\Amap^*$ is the adjoint of the linear map $\Amap$,
	and the cone $\cone^*$ is the dual cone of $\cone$, \ie $\cone^*\defn \{s\in \EE\mid \inprod{s}{x}\geq 0,\,\forall x\in \cone\}$.
	Let us introduce strong duality first.
	\begin{definition}[Strong duality]
		The primal and dual problems \eqref{opt: primalcone} and \ref{opt: d} satisfy strong duality (SD)
		if the primal and dual solution sets $\xset, \yset$ are nonempty,
		$\xset$ is compact,
		and there exists a primal and dual solution pair $(\xsol,\ysol)\in \xset\times \yset$
		such that
		\begin{equation}\tag{SD}\label{eqn: strongduality}
			\begin{aligned}
				\inprod{\cost}{\xsol}=\inprod{\bvector}{\ysol}=\inprod{\Amap \xsol}{\ysol}.
			\end{aligned}
		\end{equation}
		Equivalently, define the slack vector $\ssol = \cost-\Amap^* \ysol$ to rewrite the equality
		\ref{eqn: strongduality} as
		\[
		0 =\inprod{c-\Amap^*(\ysol)}{\xsol}= \inprod{\ssol}{\xsol}.
		\]
	\end{definition}
	
	Note that we require the existence of primal and dual optimal solutions instead of just equality of optimal values. Strong duality in the stated form is ensured by \emph{primal and dual Slater's condition}: there is a primal and dual feasible pair $(x,y)$ with $(x,c-\Amap^*y)\in \intr (\cone)\times \intr(\cone^*)$.

	Next we state the second condition: \emph{dual strict complementarity}.
	This condition is the key to established error bounds for a variety of optimization problems \cite{drusvyatskiy2018error,zhou2017unified}.
	\begin{definition}[Dual strict complementarity (DSC)]
		Given a solution pair
		$(\xsol,\ysol)\in \xset \times \yset$,
		define the \emph{complementary face} $\Fface_{\ssol} :\,=\{x\mid \inprod{x}{\ssol}=0, \,\ssol = c-\Amap^*\ysol\}\cap \cone$. The
		solution pair $(\xsol,\ysol)$
		satisfies dual strict complementarity if
		\begin{equation}\tag{DSC}\label{eqn: strict complementarityset}
			\begin{aligned}
				\xsol \in \rel{\Fface_{\ssol}}.
			\end{aligned}
		\end{equation}
		If \eqref{opt: primalcone} and \eqref{opt: d} has one such pair, we say \eqref{opt: primalcone} and \eqref{opt: d} (or simply \eqref{opt: primalcone}) admits dual strict complementarity.
		
	\end{definition}
	Let us now unpack the definition of $\Fface_{\ssol}$ and dual strict complementarity. Also see Figure \ref{fig: 2ds3dstriccomplementarity} for
	a graphical illustration of the condition.
	
	\begin{figure}[H]
		\begin{subfigure}[(a)]{.5\textwidth}
			\centering
			\vspace{0.1 \textheight}
			\includegraphics[width=0.9\linewidth, height= 0.25\textheight]{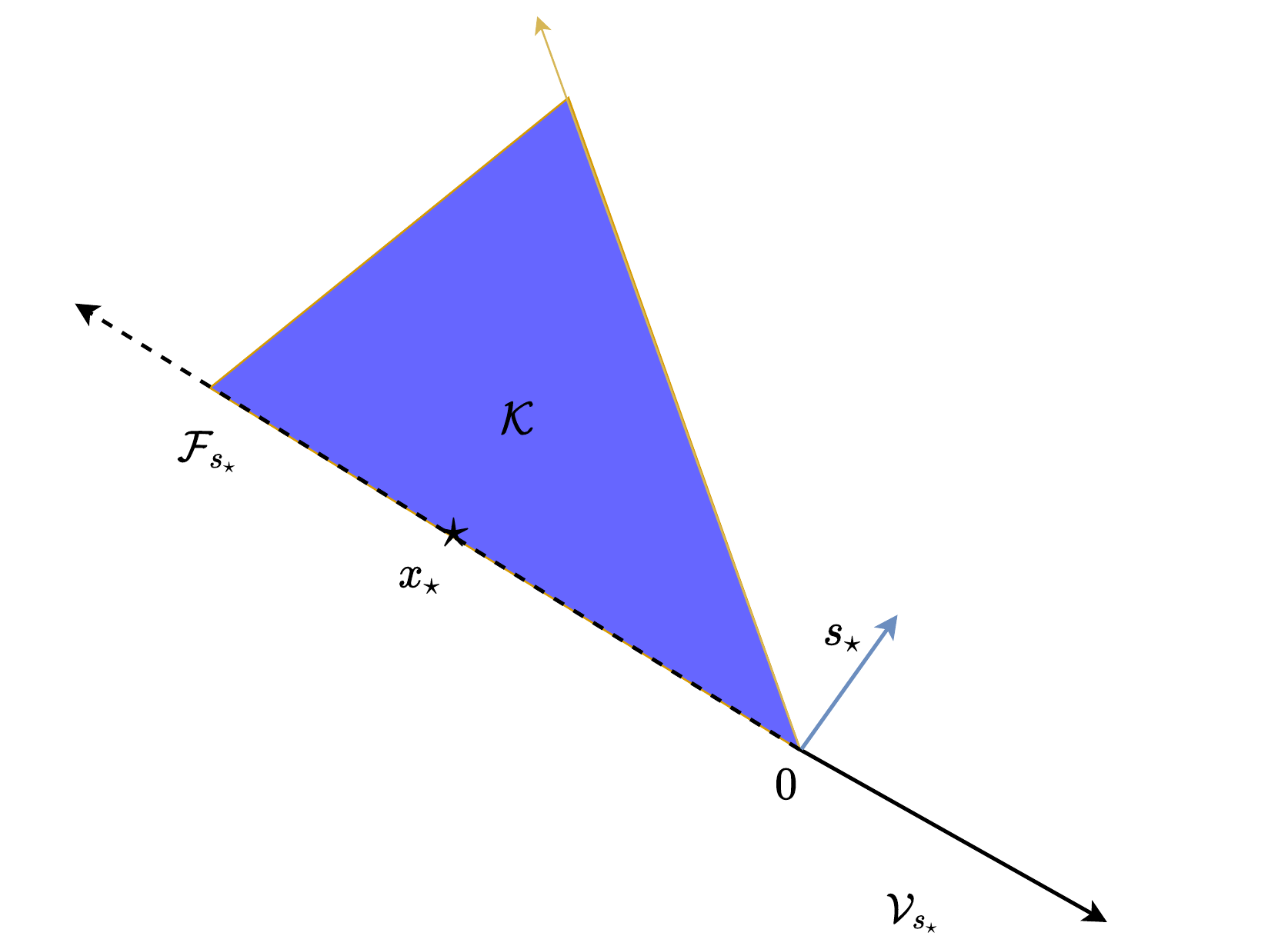}
			\caption{2D illustration}
			\label{fig:figure Matrix Completion}
		\end{subfigure}%
		\hspace{5pt}
		\begin{subfigure}[(b)]{.5\textwidth}
			\centering
			\includegraphics[width=0.7\linewidth, height= 0.35\textheight]{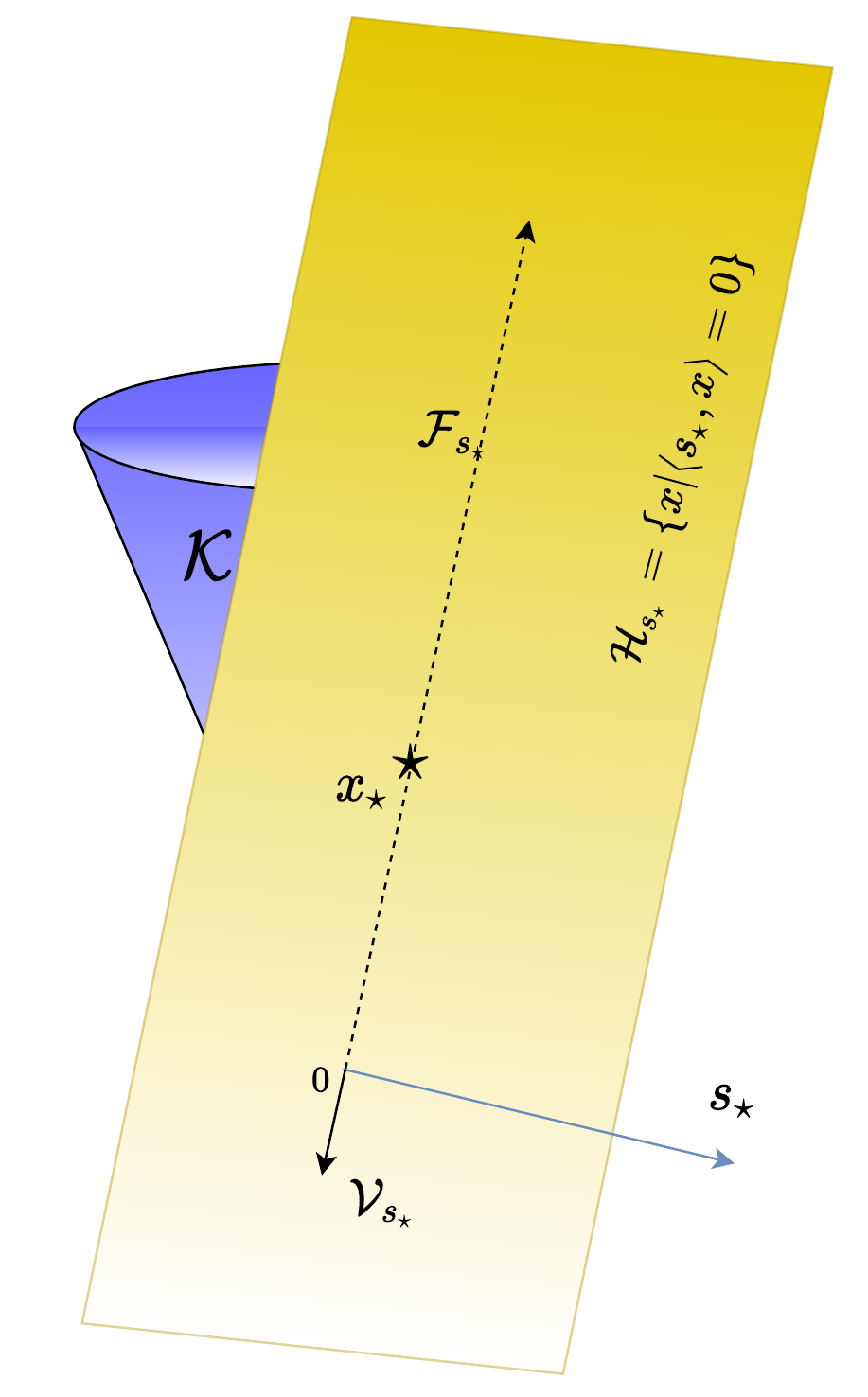}
			\caption{3D illustration}
			\label{fig:figure Max-cut}
		\end{subfigure}
		\caption{Strict Complementarity:
			For both plots, the indigo cone is the cone $\cone$;
			the slack vector $\ssol$ is the blue ray;
			the complementary face $\Fface_{\ssol}$ is the dashed black ray;
			and the complementary space $\Vspace_{\ssol}$ is the black line (both solid and dashed parts).
			In the 2D case, the complementary hyperplane $\hyperplane_{\ssol}$ and $\Vspace_{\ssol}$ coincide.
			In the 3D case, the complementary hyperplane $\hyperplane_{\ssol}$ is the yellow plane,
			which is tangent to the purple cone $\mathcal K$ at the point $\xsol$ and is
			orthogonal to $\ssol$.
			\label{fig: 2ds3dstriccomplementarity}
		}
	\end{figure}
	
	\paragraph{Understanding the complementary face $\Fface_{\ssol}$.}
	To understand the name complementary face,
	let us introduce the \emph{complementary hyperplane}, the hyperplane
	$\hyperplane_{\ssol} = \{x\mid \inprod{x}{\ssol}=0\}$
	orthogonal to the slack vector $\ssol=c-\Amap^*\ysol$.
	The complementary face $\Fface_{\ssol}$
	is simply the intersection of $\hyperplane_{\ssol}$ and the cone $\cone$.
	The intersection is nonempty due to strong duality.
	We can see that $\Fface_{\ssol}$ is indeed a face
	because its intersection with $\cone$ is nonempty (it contains $\xsol$),
	and every $x \in \cone$ lies on the same side of
	the hyperplane $\hyperplane_{\ssol}$, $\inprod{x}{\ssol}\geq 0$,
	as $\ssol\in \cone^*$.
	In particular, we see the face $\Fface_{\ssol}$ is exposed.
	
	\paragraph{Dual strict complementarity (DSC) and Slater's condition.}
	To better understand dual strict complementarity,
	define the \emph{complementary space} as the affine hull of $\Fface_{\ssol}$, $\Vspace_{\ssol}:\,=\affine{\Fface_{\ssol}}$.
	The complementary face $\Fface_{\ssol} =\hyperplane_{\ssol}\cap \cone$ is a cone,
	so the complementary space $\Vspace_{\ssol}$ is a linear subspace.
	Imagine modifying problem \eqref{opt: primalcone}
	by replacing the cone $\cone$ by $\Fface_{\ssol}$
	in problem \eqref{opt: primalcone}
	and restricting the decision variable $x$ to the subspace $\Vspace_{\ssol}$.
	Note that $\xsol$ is still a solution to this problem, so this
	procedure is related to facial reduction:
	the modified problem restricts $x$ to a face of the original cone $\cone$.
	DSC means that there is a primal $x$ in the interior
	of the cone $\Fface_{\ssol}\subset\Vspace_{\ssol}$,
	where the interior is taken w.r.t. the subspace $\Vspace_{\ssol}$.
	Hence DSC is equivalent to the usual Slater's condition for
	the modified problem. 
	
	\paragraph{Primal strict complementarity and strict complementarity.} Given dual strict complementarity (DSC),
	a natural way to define primal strict complementarity (PSC) is to
	reverse the role of $\ssol$ and $\xsol$ in the definition of DSC. Precisely,
	PSC means that there
	exists $(\xsol,\ysol)$ with $c-\Amap^*\ysol =\ssol \in \rel{\{s\mid \inprod{\xsol}{s} =0 \,\text{and}\, s\in \cone^* \}}$. Primal and
	dual strict complementarity are not always equivalent unless the cone $\cone$ is exposed.\footnote{
		For a discussion on primal and dual strict complementarity, see \cite[Remark 4.10]{dur2017genericity}}
	Happily, all the symmetric cones are exposed, including $\cone = \reals_+^\dm$, $\SOC{\dm}$, and $\sym_+^\dm$.
	\cite{dur2017genericity} shows that DSC and PSC actually hold ``generically"
	\footnote{Roughly speaking, this condition holds except on a measure $0$ set of problems
		parameterized by $\Amap,b,c$, conditioning on the existence of a primal dual solution pair.
		We refer the reader to the references for more details.} for general conic programs.
	It is worth noting that the standard notion of strict complementarity (SC) for SDP \cite[Definition 4]{alizadeh1997complementarity} and LP,
	both defined algebraically,
	are equivalent to the geometric notion of DSC here.\footnote{The definition of SC for LP and SDP, and the proof of the equivalence can be found in
		Section \ref{sec: Equivalence between DSC and SC for LP and SDP}.}
	SC always holds for LP \cite{goldman1956theory},
	holds ``generically" for SDP as shown in \cite{alizadeh1997complementarity},
	and even holds for some structured instances of SDP \cite{ding2020regularity}.
	Due to the equivalence,
	DSC also holds under the same conditions for LP and SDP.

	\subsection{Defining the strict complementary slackness approach}\label{sec: Defining The strict complementary slackness approach}
	In this section, we explain how to use the two assumptions in Section \ref{sec: assumptions}
	to establish a framework to prove error bounds of the form \eqref{eq: errorbound} and
	sensitivity bounds of the form \eqref{eq: sensitivity}.
	As we explained in the introduction, an error bound \eqref{eq: errorbound}
	can be used to derive a sensitivity bound \eqref{eq: sensitivity}.
	Hence, we focus on proving an error bound first.
	Our main theorem, Theorem \ref{thm: basicframework} in Section \ref{sec: extension},
	reduces the task to bounding the quantity
	$\twonorm{\proj{\Vspace_{\ssol}^\perp}(\conpart{x})}$.\footnote{Here $\Vspace_{\ssol}^\perp$ is orthogonal complementary space of $\Vspace_{\ssol}=\affine{\Fface_{\ssol}}$,
		and $\proj{\Vspace_{\ssol}^\perp}$ is the corresponding projection.
		Recall the conic part $\conpart{x}$ is $\conpart{x} =\proj{\cone}(x)$.}
	We explain how to further bound this quantity in Section \ref{sec: conicexamples}.

	\subsubsection{From optimality to feasibility: linear regularity of convex sets}\label{sec: linearRegularity}
	Our first step is to identify problem \eqref{opt: primalcone}
	with the feasibility problem of finding $x$ such that
	\begin{equation}\label{eq: feasibilityTerm}
		\inprod{\cost}{x}=\pval,\quad \Amap(x)=\bvector, \quad\text{and}\quad x \in \cone.
	\end{equation}
	This transformation activates the following geometric result,
	called linear regulairty regularity of convex sets {\cite[Theorem 4.6]{bauschke1999strong}}, {\cite[Theorem 2.1]{zhang2000global}},
	a classical result on error bounds for feasibility systems.
	This result states that the distance to the intersection of two sets
	is bounded by the sum of distances to the two sets.
	\begin{lemma}[Linear regularity] \label{lem: linearRegularity}
		Suppose the $C\subset \EE$ is an affine space with $C=\{x\mid \mathcal{B}x = d\}$, where $\mathcal{B}:\EE \rightarrow \FF$ is a linear map,
		and $D\subset \EE$ is a closed convex cone. If
		$C\cap \intr(D)\not=\emptyset$ and $C\cap D$ is compact,
		then there are some $\gamma,\gamma'>0$ such that for all $x\in \EE$,
		\begin{equation}
			\begin{aligned} \label{eq:CDintersectionbound}
				\dist(x, C\cap D)\leq \gamma\twonorm{\mathcal{B}x-d}+\gamma'\dist(x,D).
			\end{aligned}
		\end{equation}
	\end{lemma}
	Now it is tempting to set $C=\{x\mid \inprod{\cost}{x}=\pval,\,\Amap(x)=\bvector\}$ and $D=\cone$, and conclude \eqref{eq: errorbound} holds with
	exponent $p=1$, since $C\cap D = \xset$.
	The catch is that for most problems of interest,
	the optimal solution $\xsol \in \partial \cone$ lies on the boundary of the cone $\cone$,
	and the condition $\intr(D)\cap C=\emptyset$ \emph{does not} hold!
	Indeed, unless $c=0$, which makes \eqref{opt: primalcone} a feasibility problem,
	we always have $\xsol \in \partial \cone$.
	
	\subsubsection{Facial reduction}\label{sec: FacialReduction}
	
	We may still use Lemma \ref{lem: linearRegularity} to establish an error bound.
	The key is to use the facial reduction idea mentioned earlier.
	Recall the condition required is
	$C\cap \intr(D)= \{x\mid \inprod{\cost}{x}=\pval,\, \Amap(x)=\bvector \}\cap\intr(\cone)\not=\emptyset$, which
	does not hold for \eqref{opt: primalcone} with nonzero $\cost$.
	The problem is that the cone $\cone$ lies in the large space $\EE$,
	so its interior (with respect to $\EE$) does not contain $\xsol$.
	Instead, consider restricting the variable $x$
	to the complementary space $\Vspace_{\ssol}=\affine{\Fface_{\ssol}}$
	and replacing the cone $\cone$ by the complementary face $\Fface_{\ssol}$.
	The interior of the cone $\Fface_{\ssol}$ with respect to the space $\Vspace_{\ssol}$
	does contain $\xsol$ under strict complementarity, so we may activate Lemma \ref{lem: linearRegularity}.
	
	This modification enables an error bound for $x\in \Vspace_{\ssol}$ as stated in
	Lemma \ref{lem: linearegularitystrictcomplementarity} below.
	We also provide a more concrete estimate of the constants
	$\gamma,\gamma'$ when $\xset$ is a singleton.
	
	\begin{lemma}\label{lem: linearegularitystrictcomplementarity}
		Suppose strong duality and dual strict complementarity hold. Then  there are constants $\gamma,\gamma'$ such that for any  $x\in \Vspace_{\ssol}$:
		\begin{equation}
			\begin{aligned}\label{eqn: erbstrictcomplslack}
				\dist(x,\xset) & \leq \gamma\twonorm{\Amap(x)-b} +\gamma' \dist(x,\Fface_{\ssol}).
			\end{aligned}
		\end{equation}
		Moreover, if $\xset$ is a singleton, then we may take $\gamma'=0$ and $\gamma=\frac{1}{\sigma_{\min}(\Amap_{\Vspace_{\ssol}})}$.
		Here, the linear map $\Amap_{\Vspace_{\ssol}}$ is
		$\Amap$ restricted to $\Vspace_{\ssol}$ and
		$\sigma_{\min}(\Amap_{\Vspace_{\ssol}})$ is the smallest singular value of  $\Amap_{\Vspace_{\ssol}}$.
	\end{lemma}
	\begin{proof}
		The inequality \ref{eqn: erbstrictcomplslack} is immediate by using Lemma \ref{lem: linearRegularity} with
		$C=\{x\in \Vspace_{\ssol}\mid \Amap(x)=b\}\subset \Vspace_{\ssol}$, $D=\Fface_{\ssol}$, and $\EE = \Vspace_{\ssol}$.
		Indeed, this choice gives $C\cap D = \xset$.
		The condition $C\cap \intr(D) \not=\emptyset$, where the interior is taken with respect to the space $\Vspace_{\ssol}$, is exactly \eqref{eqn: strict complementarityset}: $\exists \xsol\in \xset$ such that $\xsol \in \relint(\Fface_{\ssol})$.
		
		Now we show that $\gamma'=0$ and $\gamma=\frac{1}{\sigma_{\min}(\Amap_{\Vspace_{\ssol}})}$ when $\xset$ is a singleton.
		First assume $\Amap_{\Vspace_{\ssol}}$ has a trivial nullspace,
		so $\xsol$ is the only solution in $\Vspace_{\ssol}$ to $\Amap_{\Vspace_{\ssol}}(x)=b$.
		Hence $\sigma_{\min}(\Amap_{\Vspace_{\ssol}})>0$ and so
		for any $x\in \Vspace_{\ssol}$,
		$\twonorm{x-\xsol}\leq \frac{1}{\sigma_{\min}(\Amap_{\Vspace_{\ssol}})}\twonorm{\Amap_{\Vspace_{\ssol}}(x-\xsol)} =  \frac{1}{\sigma_{\min}(\Amap_{\Vspace_{\ssol}})}\twonorm{\Amap(x)-b}$.
		Finally, we show by contradiction that the nullspace of $\Amap_{\Vspace_{\ssol}}$ is trivial
		whenever $\xset$ is a singleton.
		If the nullspace is not trivial,
		then there is some $x'\in \Vspace_{\ssol}$ such that $\Amap_{\Vspace_{\ssol}}(x')=0$.
		Hence $\xsol +\alpha x'$ for some small enough $\alpha$ is still optimal,
		as $\xsol \in \relint(\Fface_{\ssol})$,
		which contradicts our hypothesis that the solution set $\xset$ is a singleton.
		
	\end{proof}
	
	This choice of the face $\Fface_{\ssol}$ and
	the corresponding linear space $\Vspace_{\ssol}$
	correspond to the idea of facial reduction \cite{drusvyatskiy2017many, borwein1981facial}.
	Facial reduction is a conceptual and numerical technique designed to handle
	conic feasibility problems for which constraint qualifications
	(such as Slater's condition) fail.
	Note that such failure is the interesting case for a feasibility system \eqref{eq: feasibilityTerm} when the optimal solution $\xsol\in \partial \cone$.
	Indeed, our choice of face can be considered as one step of the facial reduction procedure.
	
	\subsubsection{Extension to the whole space: orthogonal decomposition}\label{sec: extension}
	In this section, we derive our main result, Theorem \ref{thm: basicframework},
	by extending the previous result to the whole space using the orthogonal decomposition
	$\EE = \Vspace_{\ssol}\oplus \Vspace_{\ssol}^\perp$
	with $\Vspace_{\ssol}\perp \Vspace_{\ssol}^\perp$.
	\begin{theorem}\label{thm: basicframework}
		Suppose strong duality and dual strict complementarity hold.
		Then for some
		constants $\gamma,\gamma'$ described in Lemma \ref{lem: linearegularitystrictcomplementarity} and for all $x\in \EE$,
		we have
		\begin{equation}
			\begin{aligned}\label{eqn: framework}
				\dist(x,\xset) \leq & (1+\gamma\sigma_{\max}(\Amap))\twonorm{\proj{\Vspace_{\ssol}^\perp}(\conpart{x})}+\gamma \twonorm{\Amap(x)-b}\\
				&+(1+\gamma\sigma_{\max}(\Amap))\twonorm{\proj{\Vspace^\perp_{\ssol}}(\conefeas{x})}+\gamma'\twonorm{\proj{\Vspace_{\ssol}}(\conefeas{x})}\\
				&+\gamma'\dist(\proj{\Vspace_{\ssol}}(\conpart{x}),\Fface_{\ssol}),
			\end{aligned}
		\end{equation}
		where $\proj{\Vspace_{\ssol}}$ and $\proj{\Vspace_{\ssol}^\perp}$ are orthogonal projections to $\Vspace_{\ssol}$ and $\Vspace_{\ssol}^\perp$ respectively.
		The terms $\twonorm{\proj{\Vspace_{\ssol}}(\conefeas{x})}$ and $\twonorm{\proj{\Vspace^\perp_{\ssol}}(\conefeas{x})}$ can themselves be bounded by $\twonorm{\conefeas{x}}$.
	\end{theorem}
	\begin{proof}
		Recall Lemma \ref{lem: linearegularitystrictcomplementarity} establishes an error bound
		for only those $x\in \Vspace_{\ssol}$.
		Using the orthogonal decomposition proposed, for any $x\in \EE$,
		\[
		x = \proj{\Vspace_{\ssol}}(x) + \proj{\Vspace_{\ssol}^\perp}(x).
		\]
		This decomposition immediately gives
		\begin{equation}
			\begin{aligned}\label{eqn: orthogdecomp}
				\dist(x,\xset) & \leq \twonorm{\proj{\Vspace_{\ssol}^\perp}(x)} +
				\dist(\proj{\Vspace_{\ssol}}(x), \xset).
			\end{aligned}
		\end{equation}
		The second term $\dist(\proj{\Vspace_{\ssol}}(x), \xset)$ can be bounded using Lemma \ref{lem: linearegularitystrictcomplementarity}:
		\begin{equation}\label{eqn: orthogdecompandpreviouscombine}
			\dist(\proj{\Vspace_{\ssol}}(x), \xset)\leq \gamma\twonorm{\Amap(\proj{\Vspace_{\ssol}}(x))-b}+ \gamma'\dist(\proj{\Vspace_{\ssol}}(x),\Fface_{\ssol}).
		\end{equation}
		To translate the above bound to linear infeasibility $\Amap(x)-\bvector$ and conic infeasibility $\conefeas{x}$, we note that $x= \proj{\Vspace_{\ssol}}(x)+\proj{\Vspace_{\ssol}^\perp}(x)$ and $x=\conpart{x}+\conefeas{x}$\footnote{Recall
			$\conpart{x}=\proj{\cone}(x)$ and $x=\conpart{x}+\conefeas{x}$}. With these two decompositions,
		and tad more algebra,
		we arrive at the theorem.
	\end{proof}
	
	To go further, we need to bound the following two terms in terms of $\optgap(x)$,
	$\Amap(x)-b$, and $\conefeas{x}$:
	\begin{enumerate}
		\item The distance to the space $\Vspace_{\ssol}$: $\proj{\Vspace_{\ssol}^\perp}(\conpart{x})$.
		\item The term $\dist(\proj{\Vspace_{\ssol}}(\conpart{x}),\Fface_{\ssol})$.
	\end{enumerate}
	In Section \ref{sec: conicexamples},
	we show how to bound both terms for the special cases of LP, SOCP, SDP, and
	more general conic programs \eqref{opt: primalcone} where $\cone$ is a finite product of
	LP, SOCP, or SDP cones. A quick summary of results can be found in Table \ref{tb: PAndconstantAndBoundf}.
	
	As we shall see in Section \ref{sec: conicexamples},
	the term $\dist(\proj{\Vspace_{\ssol}}(\conpart{x}),\Fface_{\ssol})$ is usually zero.
	Thus the major challenge is bounding $\twonorm{\proj{\Vspace_{\ssol}^\perp}(\conpart{x})}$\footnote{Bounding this term is in some sense necessary in establishing an error bound. See more discussion in Section \ref{sec: lowerbound} in the appendix.}.
	Note that under this condition, for feasible $x$ of \eqref{opt: primalcone}, the bound  \eqref{eqn: framework} reduces to
	\begin{equation}
		\begin{aligned}\label{eqn: frameworkSimplerSituation}
			\dist(x,\xset) & \leq (1+\gamma\sigma_{\max}(\Amap))\twonorm{\proj{\Vspace_{\ssol}^\perp}(x)}.
		\end{aligned}
	\end{equation}
	If the solution set $\xset$ is a singleton, then from Lemma \ref{lem: linearegularitystrictcomplementarity}, we know
	$\gamma=\frac{1}{\sigma_{\min}(\Amap_{\Vspace_{\ssol}})}$, and we encounter a condition number like quantity
	$ \frac{ \sigma_{\max}(\Amap)}{\sigma_{\min}(\Amap_{\Vspace_{\ssol}})}$ in \eqref{eqn: frameworkSimplerSituation}. Depending
	on applications, the condition number may scale with the problem dimension but the bound is still tight as the following example shows.
	
	\begin{example}
		Consider an SDP with $C= -\ones\ones^\top$, where $\ones \in \reals^\dm$ is the all one vector, $\Amap = \diag(\cdot)$, and
		$b = \ones$. This is a simplification of the SDP for $\mathbb{Z}_2$ synchronization \cite{bandeira2018random,ding2020regularity}. For this
		SDP, it is easily verified that the unique optimal solution is $\ones \ones^\top$ and dual strict complementarity holds with dual optimal slack $\Ssol=-\ones\ones^\top +nI$. The condition number like quantity
		$\frac{ \sigma_{\max}(\Amap)}{\sigma_{\min}(\Amap_{\Vspace_{\ssol}})}$
		in this case is $\sqrt{\dm}$ which
		does scale with the dimension $\dm$. However, if in \eqref{eqn: frameworkSimplerSituation} we let $x = I_\dm$, the identity matrix which is feasible,  then
		the LHS and RHS of \eqref{eqn: frameworkSimplerSituation}
		are $\sqrt{n^2-n}$ and $\sqrt{n-1}+\sqrt{n^2-n}$ respectively. Thus the bound is actually tight for large $\dm$.
	\end{example}

	\section{Application: error bounds}\label{sec: conicexamples}
	In this section, we show how to use the framework established in Section \ref{sec: The strict complementary slackness approach}
	to analyze conic programs \eqref{opt: primalcone} over
	the nonnegative orthant, the second order cone, the set of positive semidefinite matrices, or a finite product
	of these cones. Our analysis has two main steps:
	\begin{enumerate}
		\item Identify and write out the complementary face $\Fface_{\ssol}$ and $\Vspace_{\ssol}$.
		\item Bound the term $\twonorm{\proj{\Vspace^\perp_{\ssol}}(x_+)}$ via a function $f(\inprod{\ssol}{\conpart{x}},\norm{x})$ ,
		called the violation of complementarity,
		using the explicit structure of $\Vspace_{\ssol}$.
	\end{enumerate}
	
	We summarize the findings of this section as the following lemma and corollary.
	Refer to Table \ref{tb: PAndconstantAndBoundf} for a quick summary of the results.
	\begin{lemma}\label{lem: boundPVperpx+}
		Define the complementarity error $\epsilon(x)=\inprod{\ssol}{x}$. Suppose strong duality holds.
		The quantity $\twonorm{\proj{\Vspace^\perp_{\ssol}}(x_+)}$ can be bounded by several
		different functions $f(\inprod{\ssol}{\conpart{x}},\norm{x})$, which we call the violation of complementarity,
		depending on the slack vector $\ssol$ and the cone $\cone$.
		The first two trivial cases are the following.
		\begin{enumerate}
			\item[1.] If $\ssol =0$, then $\norm{\proj{\Vspace^\perp_{\ssol}}(x_+)}=0=:f_{0}(\epsilon(\conpart{x}))$.
			\item[2.] If $\ssol \in \intr(\cone^*)$, then $\twonorm{\proj{\Vspace^\perp_{\ssol}}(x_+)}=\twonorm{x_+}\leq c_\star \epsilon(\conpart{x}) =:f_{\intr}(\epsilon(\conpart{x}))$,
			where $c_\star = \sup_{x\in \cone}\frac{1}{\inprod{\ssol}{\frac{x}{\twonorm{x}}}}<\infty$.
		\end{enumerate}
		Moreover, for the nontrivial case $\ssol\in \partial \cone^*/ \{0\}$, we have the following bounds:
		\begin{enumerate}
			\item[3.] $\cone = \reals_+^\dm$: $\twonorm{\proj{\Vspace^\perp_{\ssol}}(x_+)}\leq \frac{1}{s_{\min>0}}\epsilon(\conpart{x})=:f_{\tiny \reals_+^\dm}(\epsilon(\conpart{x})),$ where $s_{\min>0}$ is the smallest nonzero element of $\ssol$.
			\item[4.] $\cone = \SOC{\dm}$: $\twonorm{\proj{\Vspace_{\ssol}^\perp}(x_+)}
			\leq\sqrt{2\sqrt{2}\frac{\twonorm{x}\epsilon(\conpart{x})}{\twonorm{\ssol}}} =:f_{\tiny \SOC{\dm}}(\epsilon(\conpart{x}),\twonorm{x})$.
			\item[5.] $\cone = \sym_+^\dm$:	$\fronorm{\proj{\Vspace_{\Ssol}^\perp}(\conpart{X}) } \leq  \frac{\epsilon(\conpart{X})}{T} +
			\sqrt{2\frac{\epsilon(\conpart{X})}{T} \opnorm{X}}=:f_{\tiny \sym_+^{\dm}}(\epsilon(\conpart{X}),\opnorm{X})$. Here $T$ is the smallest nonzero eigenvalue of $\Ssol$.
		\end{enumerate}
	\end{lemma}
	\begin{proof}
		Let us first consider the two trivial cases (1) $\ssol =0$ and (2) $\ssol \in \intr(\cone^*)$.
		These cases are excluded whenever $\cost$ and $\bvector$ are both nonzero.
		In the first case, $\Vspace^\perp_{\ssol}=\{0\}$,
		and we simply have $\norm{\proj{\Vspace^\perp_{\ssol}}(x_+)}=0$.
		In the second case, we have $\Vspace^\perp_{\ssol}=\EE$, and
		$\norm{\proj{\Vspace^\perp_{\ssol}}(x_+)}=\norm{x_+}\leq c_\star \inprod{\ssol}{x_+}$,
		where $c_\star = \sup_{x\in \cone}\frac{1}{\inprod{\ssol}{\frac{x}{\norm{x}}}}<\infty$.
		We defer the proof for the other cases to Section \ref{sec: LP}, \ref{sec: SOCP},
		and Section \ref{sec: SDP} for LP, SOCP, and SDP respectively.
	\end{proof}
	
	Combining Lemma \ref{lem: boundPVperpx+}, and Theorem \ref{thm: basicframework}, we reach the following corollary.
	The quantity $\dist(\proj{\Vspace_{\ssol}}(\conpart{x}),\Fface_{\ssol})$ can be verified to be zero for
	the two trivial cases by noting (i) the complementary space $\Vspace_{\ssol}=\EE$ and $\Fface_{\ssol}=\cone$ for the case $\ssol=0$, and (ii) the complementary space $\Vspace_{\ssol}=\{0\}$ and $\Fface_{\ssol}$ is a closed cone for the case $\ssol \in \intr(\cone^*)$.
	It is also zero for other three cases as shown in Sections \ref{sec: LP}--\ref{sec: SDP}.
	\begin{corollary}\label{cor: errorboundLPSOCPSDPandOthers}
		Suppose strong duality and dual strict complementarity hold,
		and one of the five cases in Lemma \ref{lem: boundPVperpx+} pertains.
		Then there exists constants $\gamma,\gamma'$ so that for all $x\in \EE$,
		\begin{equation}
			\begin{aligned}\label{eqn: errorboundLPSOCPSDPandOthers}
				\dist(x,\xset) & \leq\kappa f(\epsilon(\conpart{x}),\norm{x})+\gamma \twonorm{\Amap(x)-b}\\
				&+\kappa \twonorm{\proj{\Vspace^\perp_{\ssol}}(\conefeas{x})}+\gamma'\twonorm{\proj{\Vspace_{\ssol}}(\conefeas{x})},\\
			\end{aligned}
		\end{equation}
		where the condition number $\kappa = 1+\gamma\sigma_{\max}(\Amap)$.
		In particular, when $\xset$ is a singleton, then $\gamma'=0$ and $\gamma = \frac{1}{\sigma_{\min}(\Amap_{\Vspace_{\ssol}})}$.
		Here the formula for $f(\epsilon(x),\norm{x})$ can be found in Lemma \ref{lem: boundPVperpx+} for each of the different cases, and we can further decompose the complementarity error
		$\epsilon(\conpart{x})=\optgap(x) +\inprod{\ysol}{b-\Amap(x)}- \inprod{\ssol}{\conefeas{x}}$ using $x=\conpart{x}+\conefeas{x}$.
	\end{corollary}
	A few remarks regarding the lemma and the corollary are in order.
	
	\begin{remark}[Global and local error bound]\label{rmk: localSOCPbound}
		Note the formula for the violation of complementarity  $f(\epsilon(x),\norm{x})$
		uses $\twonorm{x}$ for SOCP and $\opnorm{X}$ for SDP.
		Hence the bound \eqref{eqn: errorboundLPSOCPSDPandOthers} for these two cases
		does not quite align with the form of the error bound \eqref{eq: errorbound} we seek.
		To eliminate the dependence on this norm (by bounding the norm),
		either of the following conditions suffices:
		\begin{itemize}
			\item  $\norm{x}\leq B$ for some constant $B$,
			\item  $\max\{\optgap(x)\,,\norm{\Amap x-b}\,,\norm{\conefeas{x}}\}\leq \bar c$ for some constant $\bar c$.
		\end{itemize}
		The second requirement combined with \eqref{eqn: errorboundLPSOCPSDPandOthers}
		for SOCP (SDP) implies $\twonorm{x}$ ($\opnorm{X}$) is bounded by some $\bar B$ depending on $\bar c$ but independent of $x$ ($X$).
		We may then replace the term $\twonorm{x}$ ($\opnorm{X}$) by $B$ or $\bar B$ in $f$.
		Requiring either of these two conditions produces a \emph{local} error bound.
		Interestingly, no such condition on the norm is not necessary for the LP case and the other two trivial cases;
		hence the bounds \eqref{eqn: errorboundLPSOCPSDPandOthers} in these cases are \emph{global} error bounds.
	\end{remark}
	
	\begin{remark}[Value of $p$ and estimate of $c_i$ in \eqref{eq: errorbound}]
		Ignoring the term $f$, the bound \eqref{eqn: errorboundLPSOCPSDPandOthers} in Lemma \ref{lem: boundPVperpx+} is linear in
		$\optgap(x),\norm{\Amap x-b},\norm{\conefeas{x}}$.
		For LP and the two trivial cases,
		$f$ is linear in $\epsilon(\conpart{x})$,
		hence the error bound \eqref{eq: errorbound} holds with exponent $p=1$.
		For SOCP and SDP, the square root of $\epsilon(\conpart{x})$ appears in $f$,
		hence \eqref{eqn: errorboundLPSOCPSDPandOthers} gives an error bound
		of the form \eqref{eq: errorbound} with exponent $p=2$,
		under the assumption $\norm{x}\leq B$.
		
		Now let's consider the constants $c_1$, $c_2$, and $c_3$ in the error bound \eqref{eq: errorbound}. It is cumbersome to estimate these for general $x$;
		here, suppose $x$ is feasible.
		For the SOCP and SDP cases, also suppose $\norm{x}\leq B$.
		Then the bound \eqref{eqn: errorboundLPSOCPSDPandOthers} reduces to
		\begin{equation}
			\begin{aligned}\label{eqn: errorboundLPSOCPSDPandOthersFeasible}
				\dist(x,\xset) & \leq (1+\gamma\sigma_{\max}(\Amap))f(\optgap(x),B).
			\end{aligned}
		\end{equation}
		The resulting constant $c_1$ for $\optgap(x)$ in the error bound \eqref{eq: errorbound}
		for each of the five cases appears in Table \ref{tb: PAndconstantAndBoundf}.
	\end{remark}
	\begin{table}
		\centering
		\begin{tabular}{cccccc}
			\hline
			Conic                                      & $\ssol=$ & $\ssol \in  $ & \multirow{2}{*}{LP}  & \multirow{2}{*}{SOCP} &  \multirow{2}{*}{SDP} \\
			program                                         & $0$ & $  \intr(\cone^*)$ &   & & \\
			\hline
			violation of                    	& \multirow{3}{*}{0}
			&  \multirow{3}{*}{$c_\star \epsilon(\conpart{x})$} &  \multirow{3}{*}{$\frac{\epsilon(\conpart{x})}{s_{\min>0}}$}
			&   \multirow{3}{*}{$\sqrt{{\frac{2\sqrt{2}\twonorm{x}\epsilon(\conpart{x})}{\twonorm{\ssol}}}}$}
			&  \multirow{3}{*}{$\frac{\epsilon(\conpart{X})}{T} +
				\sqrt{\frac{2\epsilon(\conpart{X})\opnorm{X}}{T}}$}  \\
			complementarity               &  &  &   & & \\
			$f(\epsilon(\conpart{x}),\norm{x})$    & &  &   & & \\
			\hline
			exponents $p,p'$ in	            &  \multirow{2}{*}{1}  &  \multirow{2}{*}{1} &  \multirow{2}{*}{1}  &  \multirow{2}{*}{2} &  \multirow{2}{*}{2} \\
			\eqref{eq: errorbound} and \eqref{eq: sensitivity} &   & & & & \\
			\hline
			constant $c_1$ for $\optgap(x)$	                    & 0  & $\kappa c_\star $ & $\frac{\kappa}{s_{\min>0}}$ & $2\sqrt{2}\kappa^2\frac{B}{\twonorm{\ssol}}$ &
			$\kappa^2\frac{8B}{T}$ \\
			\hline
		\end{tabular}
		\caption{This table presents the bound $f(\epsilon(\conpart{x}),\norm{x})$  for $\twonorm{\proj{\Vspace_{\ssol}}(\conpart{x})}$,
			the power $p$ in \eqref{eq: errorbound} and \eqref{eq: sensitivity}, and an estimate of $c_1$ for \emph{feasible} $x$ for different cases based on $\ssol$ and $\cone$.
			For $\ssol \in \intr(\cone*)$, the quantity $c_\star = \sup_{x\in \cone}\frac{1}{\inprod{\ssol}{\frac{x}{\norm{x}}}}<\infty$.
			For LP, the quantity $s_{\min>0}$ is the smallest nonzero element of $\ssol$.
			For SDP, the quantity $T$ is the smallest nonzero eigenvalue of $\Ssol$.
			The condition number $\kappa = 1+\gamma\sigma_{\max}(\Amap)$ and is $1+\frac{\sigma_{\max}(\Amap)}{\sigma_{\min}(\Amap_{\Vspace_{\ssol}})}$ when
			$\xset$ is a singleton. We assume $\twonorm{x}\leq B$ for SOCP and $\opnorm{X}\leq B$ for SDP. We also assume $\frac{\epsilon(\conpart{X})}{T} \leq
			\sqrt{2\frac{\epsilon(\conpart{X})B}{T}}$ for SDP.
		} \label{tb: PAndconstantAndBoundf}
	\end{table}
	
	\begin{remark}[Conditions for LP]\label{rmk: LPcase}
		Recall that for linear programming,
		dual strict complementarity is the same as strict complementarity, which always holds under strong duality \cite{goldman1956theory}.
		Hence we need not explicitly require the dual strict complementarity condition.
		Also the compactness condition for strong duality in Section \ref{sec: assumptions}
		can be dropped if we establish Theorem \ref{thm: basicframework} using Hoffman's lemma \cite{hoffman1952approximate} instead of Lemma \ref{lem: linearRegularity}.
	\end{remark}

	\begin{remark}[Finite product of cones]\label{rmk: ProductOfCones}
		Error bounds for a conic program whose cone is a finite product
		of $\reals_+^\dm, \SOC{\dm}$, and $\sym_+^\dm$ can be
		established by bounding the term $\twonorm{\proj{\Vspace_{\ssol}}(\conpart{x})}$
		by a sum of te correponding $f$s in	Lemma \ref{lem: boundPVperpx+}.
		We omit the details.
	\end{remark}
	
	Next, we prove the bound $f$, violation of complementarity, in Lemma \ref{lem: boundPVperpx+}
	for the LP, SOCP, and SDP comes by following the aforementioned procedure:
	(i) identify and write out $\Fface_{\ssol}$ and $\Vspace_{\ssol}$, and (ii) bound the term $\twonorm{\proj{\Vspace^\perp_{\ssol}}(x_+)}$.
	
	\subsection{Linear programming (LP)}\label{sec: LP}
	In linear programming, the cone $\cone=\reals_+^\dm = \{x\in \reals^\dm \mid x_i\geq 0,\,\forall \,i =1,\dots, \dm\}$.
	
	\paragraph{Identify $\Fface_{\ssol}$ and $\Vspace_{\ssol}$.}For a particular dual optimal solution $(\ysol,\ssol)$, satisfying dual strict complementarity, the complementary face
	$\Fface_{\ssol}=\{x \in \reals_+^\dm \mid x_i = 0,\text{ for all } (\ssol)_i>0\}$, and the complementary space  $\Vspace_{\ssol}=\{x\in \reals^\dm \mid x_i = 0 \text{ for all } (\ssol)_i>0\}$.
	Hence, the term $\dist(\proj{\Vspace_{\ssol}}(x_+), \Fface_{\ssol})$ is simply zero as $\proj{\Vspace_{\ssol}}(x_+)\in \Fface_{\ssol}$.
	
	\paragraph{Bound the term $\twonorm{\proj{\Vspace^\perp_{\ssol}}(x_+)}$.} For the term $\twonorm{\proj{\Vspace^\perp_{\ssol}}(x_+)}$, denote $I_{\ssol} =\{i\mid (\ssol)_i>0\}$, $I^c_{\ssol}=\{1,\dots,n\}-I_{\ssol}$, and $s_{\min>0} = \min_{i\in I_s} s_i$, we have
	\begin{equation}\label{eqn: LPdistanceToinnerproduct}
		\begin{aligned}
			\twonorm{\proj{\Vspace^\perp_{\ssol}}(x_+)}=\twonorm{(x_+)_{I_{\ssol}}} & \leq  \onenorm{(x_+)_{I_{\ssol}}}
			&\leq \frac{1}{s_{\min>0}}\inprod{\ssol}{x_+}.\\
		\end{aligned}
	\end{equation}
	Hence, Lemma \ref{lem: boundPVperpx+} for the LP case is established.
	\subsection{Second order programming (SOCP)}\label{sec: SOCP}
	In second order cone programming, the cone $\cone$ is $\SOC{\dm} = \{x=(x_{1:n},x_{n+1})\mid \twonorm{x_{1:n}}\leq x_{n+1} ,\, x_{1:n}\in \reals^\dm, \, x_{n+1} \in \reals\}$.
	
	\paragraph{Identify $\Fface_{\ssol}$ and $\Vspace_{\ssol}$.} Given a dual solution $\ssol = (\ssoll{1:n},\ssoll{n+1})\in (\SOC{\dm})^* = \SOC{\dm}$ satisfying dual strict complementarity,
	the complementary face is defined as $\Fface_{\ssol} = \{x\mid \inprod{\ssoll{1:n}}{x_{1:n}}+x_{n+1}\ssoll{n+1}=0,; x\in \SOC{\dm}\}$.
	We can further simplify this expression as
	\[
	\Fface_{\ssol} = \{\lambda (-\ssoll{1:n}, \ssoll{n+1})\mid \lambda\geq 0\}.
	\]
	The complementary space $\Vspace_{\ssol}$ for the nontrivial case
	$\ssol \in \partial \SOC{\dm} \setminus (0,0)$
	is simply
	\[\Vspace_{\ssol}=\Span\{ (-\ssoll{1:n},\ssoll{n+1})\}.
	\]
	
	For the dual optimal solution $\ssol$, denote $\checkssol=\frac{1}{\norm{\ssol}}(-\ssoll{1:n},\ssoll{n+1})$. Note that  the term $\dist(\proj{\Vspace_{\ssol}}(x _+), \Fface_{\ssol} )$ is again simply zero as $\proj{\Vspace_{\ssol}}(x_+) \in \Fface_{\ssol}$.
	
	\paragraph{Bound the term $\twonorm{\proj{\Vspace^\perp_{\ssol}}(x_+)}$.}
	We now turn to analyze $\proj{\Vspace^\perp}(x_+)$, which can be written explicitly as
	\begin{equation}
		\begin{aligned}
			\proj{\Vspace^\perp}(x_+) & =
			x_+- \inprod{x_+}{\checkssol}\checkssol.
		\end{aligned}
	\end{equation}
	Now introduce the shorthand $\bar{\epsilon}(\conpart{x}) = \inprod{x_+}{\frac{\ssol}{\norm{\ssol}}}= \inprod{x_{+,1:\dm}}{\frac{\ssoll{1:n}}{\norm{\ssol}}}+x_{+,n+1}\frac{\ssoll{n+1}}{\norm{\ssol}}$. The norm square of $\proj{\Vspace^\perp}(\bfx_+)$ can be written as
	\begin{equation}\label{eqn: SOCPdistanceToinnerproduct}
		\begin{aligned}
			\twonorm{\proj{\Vspace^\perp}(x_+)}^2 & =  \twonorm{x_+}^2 - \inprod{x_+}{\checkssol}^2 \\
			& =  \twonorm{x_+}^2 +x_{+,\dm+1}^2 -\left(-\inprod{x_{+,1:n}}{\ssoll{1:n}} +x_{+,\dm+1}\frac{\ssoll{n+1}}{\norm{\ssol}}\right)^2 \\
			& \overset{(a)}{\leq} 2x_{+,\dm+1}^2 -(2x_{+,\dm+1}\frac{\ssoll{n+1}}{\norm{\ssol}}-\bar{\epsilon}(\conpart {x}))^2  
			\overset{(b)}{\leq}  -\bar{\epsilon}^2(\conpart{x}) +2\sqrt{2}x_{+,\dm+1}\bar{\epsilon}(\conpart{x}),
		\end{aligned}
	\end{equation}
	where in step $(a)$, we use the fact that $x_+\in \SOC{\dm}$ and the definition of $\epsilon$. In step $(b)$, we use the fact $\ssol\in \partial{\SOC{\dm}}$.
	Lemma \ref{lem: boundPVperpx+} for the SOCP case is established by noting $x_{+,\dm+1}\leq \norm{\conpart{x}}\leq \norm{x}$.
	
	\subsection{Semidefinite programming (SDP)}\label{sec: SDP}
	
	In semidefinite programming, the cone $\cone$ is $\sym_+^\dm =\{X\in \sym^\dm\mid X\succeq 0 \}$. Note that we use capital letter $X$ and $S$ for matrices. 
	
	\paragraph{Identify $\Fface_{\ssol}$ and $\Vspace_{\ssol}$.}
	For a dual optimal solution $(\ysol,S_\star)$ with
	$S_\star = C-\Amap \ysol \succeq 0$ satisfying dual strict complementarity,
	the complementary face $\Fface_{S_\star} :=
	\{X\mid \inprod{X}{S_\star}=0, X\succeq 0\} = \{X\mid X = VRV^\top, R\in \sym^r,R\succeq 0\}$, where $\rank(S_\star)=n-r$,
	and $V\in \reals ^{\dm \times r}$ is a matrix with orthonormal columns that span $\range(S_\star)$.
	In this case, the complementary space  $\Vspace_{S_\star} = \{VRV^\top \mid R \in \sym ^r\}$.
	We note that  the term $\dist(\proj{\Vspace_{\Ssol}}(X_+), \Fface_{S_\star})$
	is again zero as $\proj{\Vspace_{\Ssol}}(X_+) \in \Fface_{S_\star}$.
	
	\paragraph{Bound the term $\fronorm{\proj{\Vspace^\perp_{\ssol}}(X_+)}$.}
	We now turn to bound the term $\fronorm{\proj{\Vspace_{\Ssol}^\perp}(X_+)}$.
	We utilize Lemma \ref{lem: XSsmallXliveinS} \cite[Lemma 4.3]{ding2019optimal}
	to bound the term
	$\fronorm{\proj{\Vspace_{\Ssol}^\perp}(X_+)}$
	with $S = \Ssol$, and use the observation that $\opnorm{X_+}\leq \opnorm{X}$.
	\begin{lemma}\label{lem: XSsmallXliveinS}
		Suppose $X,S\in \sym^{\dm}$ are both positive semidefinite.
		Let $V\in \reals^{\dm\times r}$ be the matrix formed by the eigenvectors
		with the $r$ smallest eigenvalues of $S$ and define $\Vspace = \range(V)$.
		Let $\epsilon =\tr(XS)$.
		If $T = \lambda_{n-r}(S)>0$, then
		\[
		\fronorm{\proj{\Vspace^\perp}(X) } \leq  \frac{\epsilon}{T} +
		\sqrt{2\frac{\epsilon}{T} \opnorm{X}}
		\quad\text{and}\quad
		\nucnorm{\proj{\Vspace}(X)
		} \leq  \frac{\epsilon}{T} + 2\sqrt{r\frac{\epsilon}{T} \opnorm{X}}.
		\]
	\end{lemma}
	
	\section{Application: sensitivity of solution}\label{sec: fromerrorboundtosensitivityofsolution}
	As discussed in the introduction, to study the sensitivity of the solution,
	we consider a solution $\xsol'$ of the perturbed problem
	\begin{equation}\tag{$\mathcal{P'}$}
		\begin{aligned}\label{opt: primalconeperturbed}
			& \underset{x}{\text{minimize}}
			& & \inprod{\cost'}{x} \\
			& \text{subject to} & &  \Amap' x = \bvector' ,\\
			& & & x\in \cone.
		\end{aligned}
	\end{equation}
	where the problem data $(\Amap',\bvector',\cost')= (\Amap,\bvector,\cost) + (\DeltaA, \Deltab,\Deltac)$ for some small perturbation $\Delta = (\DeltaA, \Deltab,\Deltac)$,
	and ask how the distance $\twonorm{\xsol'-\xsol}$ changes according to $\Delta$.
	
	Note that once the error bound \eqref{eq: errorbound} is established,
	we can understand the sensitivity of the solution by estimating
	the suboptimality, linear and conic infeasibility of the new solution $\xsol'$
	with respect to the original problem \eqref{opt: primalcone}
	via the perturbation $\Delta$.
	Following this strategy, we prove the following theorem:
	\begin{theorem}\label{thm: sensitivityOfsolution}
		Suppose the primal and dual Slater's condition holds for some $(x_0,y_0)$, and the map $\Amap$ is surjective. For any small enough $\varepsilon>0$, there is some constant $\bar{c}$, such that for any optimal $\xsol'$ of \eqref{opt: primalconeperturbed}, and all $\norm{\Delta}:= \opnorm{\DeltaA} +\twonorm{\Deltab}+\twonorm{\Deltac}\leq \varepsilon$, we have
		\[
		\max\{\optgap(\xsol'),\twonorm{\Amap(\xsol')-\bvector}\}\leq \bar{c}\norm{\Delta}.
		\]
		Hence if \eqref{eq: errorbound} holds, then \eqref{eq: sensitivity} holds with $p=p'$.
	\end{theorem}
	To facilitate the proof, we define the smallest nonzero singular value of $\Amap$ as
	$\sigma_{\min>0}(\Amap)= \min_{\twonorm{x}=1,\\ x\perp \nullspace(\Amap)}\twonorm{\Amap(x)}$,
	and the pseudoinverse of
	$\Amap'$ as $(\Amap')^\dagger (y) = \argmin_{\Amap'(x)=y}\twonorm{x}$.
	\begin{proof}
		Consider any solution $\xsol'$ to the problem \eqref{opt: primalconeperturbed}.
		Suppose the following assumptions are satisfied (proved in Appendix \ref{sec: apprndixfromerrorboundtosensitivityofsolution}):
		\begin{enumerate}
			\item \emph{Primal and dual Slater's condition for \eqref{opt: primalconeperturbed}}:
			there exist a primal $x_0'$ and dual $y_0'$ solution feasible for problem
			\eqref{opt: primalconeperturbed} and its dual
			that satisfy $\min\{\dist(x_0',\partial \cone),\dist(c-\Amap^*y_0',\partial(\cone^*))\}>\rho$ for some $\rho$ independent of $\Delta$,
			and $\max \{\twonorm{x_0-x_0'},\twonorm{y_0-y_0'}\}<\xi$ for some $\xi$ independent of $\Delta$.
			\item \emph{Boundedness of primal solutions of \eqref{opt: primalconeperturbed}}: There is some $B>0$ independent of $\Delta$ such that any solution $\xsol$ to \eqref{opt: primalconeperturbed}
			satisfies $\twonorm{\xsol}\leq B$.
		\end{enumerate}
		Let us start with the linear infeasibility $\Amap (\xsol')-b$. Using the linear feasibility of $\xsol'$, $\Amap'\xsol'=b'$, w.r.t. \eqref{opt: primalconeperturbed}, we have
		\begin{equation}
			\begin{aligned}\label{eqn: lfeasperturbed}
				\Amap'(\xsol') = b' & \implies (\Amap +\DeltaA)\xsol' = b+\Deltab \implies \Amap\xsol'-b =\Deltab -\DeltaA \xsol' \\
				&\implies \twonorm{\Amap \xsol'-b}\leq \twonorm{\Deltab}+\opnorm{\DeltaA}B.
			\end{aligned}
		\end{equation}
		This shows $\twonorm{\Amap \xsol'-b}\leq c\norm{\Delta}$ for any $c>B+1$.
		
		Next, consider $\optgap(\xsol')$. We would like to use the optimality of $\xsol'$ to \eqref{opt: primalconeperturbed} and compare it against $\xsol$.
		However, since
		$\xsol$ is not necessarily feasible for \eqref{opt: primalconeperturbed}, we need more subtle reasoning.
		Consider
		$\hat{x} \defn (1-\alpha)\left(\xsol + (\Amap')^\dagger(\Deltab-\DeltaA\xsol)\right)+\alpha x_0'$ with $\alpha = \frac{\twonorm{(\Amap')^\dagger(\Deltab-\DeltaA\xsol)}}{\rho +\twonorm{(\Amap')^\dagger(\Deltab-\DeltaA\xsol)}}
		$. Here $(\Amap')^{\dagger}$ exists and has largest singular value at most $\frac{2}{\sigma_{\min>0}(\Amap)}$ as long as $\sigma_{\max}({\DeltaA})\leq \frac{\sigma_{\min>0}(\Amap)}{2}$. We have
		$\twonorm{\hat{x}}\leq B_1$ for some $\Delta$ independent $B_1$ as $\twonorm{x_0-x_0'}<\xi$.  Note that  $\hat{x}\in \cone$ since
		\[
		\hat{x} = (1-\alpha) \underbrace{\xsol}_{\in \cone} + \alpha \underbrace{\left(x_0' + \frac{1-\alpha}{\alpha} (\Amap')^\dagger(\Deltab-\DeltaA\xsol)\right)}_{\in \cone \;\text{since}\;
			\twonorm{\frac{1-\alpha}{\alpha} (\Amap')^\dagger(\Deltab-\DeltaA\xsol)}\leq \rho },
		\]
		if $\alpha >0$. The case of $\alpha =0$ is trivial.
		We also know $\hat{x}$ is feasible with respect to the linear constraints of \eqref{opt: primalconeperturbed}:
		\begin{equation}
			\begin{aligned}\label{eqn: optgfeas}
				\Amap'(\hat{x})  & =(1-\alpha)\Amap'(\xsol + (\Amap')^\dagger(\Deltab-\DeltaA\xsol)) + \alpha \Amap'x_0' \\
				&= (1-\alpha)(b+\DeltaA(\xsol)+\Deltab -\DeltaA\xsol) +\alpha \bvector' = \bvector'.
			\end{aligned}
		\end{equation}
		Note by construction $\hat{x}- \xsol = \alpha (x_0' - \xsol -\alpha (\Amap')^\dagger(\Deltab-\DeltaA\xsol))=:\Delta x$
		with $\alpha = \frac{\twonorm{(\Amap')^\dagger(\Deltab-\DeltaA\xsol)}}{\rho +\twonorm{(\Amap')^\dagger(\Deltab-\DeltaA\xsol)}}
		$. Hence $\twonorm{\Delta x }\leq c'\left(\twonorm{\Deltab}+\opnorm{\DeltaA}\right)$ for some constant $c'$. Now using the optimality of $\xsol'$, we have
		\begin{equation}
			\begin{aligned}\label{eqn: optgap}
				\cost'\xsol' \leq \cost'\hat{x} \implies \optgap(\xsol') = \cost^\top \xsol' - \cost ^\top \xsol \leq \Deltac^\top (\hat{x}-\xsol')+\cost^\top \Delta x.
			\end{aligned}
		\end{equation}
		Hence $\optgap(\xsol') \leq \bar{c}\norm{\Delta}$ for large enough $\bar{c}$ using $\twonorm{\xsol'}\leq B$.
	\end{proof}
	
	\section{Discussion}\label{sec: discussion}
	Using the framework established in Section \ref{sec: The strict complementary slackness approach},
	we have shown an error bound of the form \eqref{eq: errorbound} and
	a bound on the sensitivity of the solution with respect to problem data \eqref{eq: sensitivity}
	for a broad class of problems: LP, SOCP, and SDP,
	and conic programs for which the cone is a finite product of the LP, SOCP, and SDP cones.
	Let us now compare the results we have obtained with the literature on error bounds and sensitivity of solution.
	
	\paragraph{Error bound}
	The celebrated results of \cite{hoffman1952approximate} show that the
	error bound is linear for linear programming: in \eqref{eq: errorbound},
	the exponent $p=1$.
	The work of Sturm \cite[Section 4]{sturm2000error} shows that under strict complementarity and compactness of the solution set,
	SDPs satisfy a quadratic error bound: $p=2$ in \eqref{eq: errorbound}.
	Sturm also discusses the exponent of $\dist(x,\xset)$ without strict complementarity:
	it can be bounded by $2^d$ where $d$ is the \emph{singularity degree},
	which is at most $n-1$~\cite[Lemma 3.6]{sturm2000error}. 
	A recent result shows that $p=2$  under dual strict complementarity type conditions when the cone $\cone$ is an \emph{amenable cone},
	which includes all symmetric cones~\cite{lourencco2019amenable}.
	When the cone $\cone$ is defined as a semialgebraic set (LP, SOCP, SDP are special cases),
	Drusvyatskiy, Ioffe, and  Lewis \cite[Corrollary 4.8]{drusvyatskiy2016generic} showed that for generic cost vector $c$,
	the exponent $p$ is always $2$ when the inequality is restricted to feasible $x$.
	We note that the proofs for these bounds do not provide estimates for $c_i$.
	In our framework, estimates of $c_i$ (expressed in terms of the primal and dual solution)
	can be obtained supposing the primal solution is unique.
	
	\paragraph{Sensitivity of solution}
	When $p'=1$, the bound describing the sensitivity of the solution \eqref{eq: sensitivity}
	is also called stability or metric regularity \cite[Definition A.6]{ruszczynski2011nonlinear},
	discussed in detail in \cite[Appendix A]{ruszczynski2011nonlinear},
	and \cite{klatte2006nonsmooth,bonnans2013perturbation}.
	In the context of semidefinite programming,
	when the primal and dual solutions are unique and strict complementarity holds,
	Nayakkankuppam and Overton \cite{nayakkankuppam1999conditioning} shows that \eqref{eq: sensitivity}
	holds with $p'=1$.
	When the cone $\cone$ is a semialgebraic set,
	Drusvyatskiy, Ioffe, and  Lewis \cite[section 5]{drusvyatskiy2016generic} showed that
	for generic perturbations in the cost vector and the right hand side vector $c-c',b-b'$, the sensitivity bound \eqref{eq: sensitivity} holds with $p'=1$.
	
	In Section \ref{sec: fromerrorboundtosensitivityofsolution}, we have seen
	that \eqref{eq: sensitivity} holds whenever \eqref{eq: errorbound} with $p'=p$.
	Hence, using our results from Section \ref{sec: conicexamples},
	we see $p'=1$ for LP, and $p'=2$ for SOCP and SDP.
	This result improves on the previous bound for SOCPs and SDPs assuming only strict complementarity.

	\paragraph{Other Cones?} An interesting future direction is the extension to other cones,
	\eg the copositive cone, the completely positive cone, and
	the doubly positive cone (the intersection of nonnegative matrices and positive semidefinite matrices).
	Can we still bound the term $\twonorm{\proj{\Vspace_{\ssol}^\perp }(x)}$?
	For the cones $\reals_+^\dm$, $\SOC{\dm}$, and $\sym_+^\dm$,
	our technique relies on the explicit structure of $\reals_+^\dm$, $\SOC{\dm}$, and $\sym_+^\dm$
	to bound $\twonorm{\proj{\Vspace_{\ssol}^\perp }(x)}$.
	Characterizing the facial structure seems to be challenging for other cones.
	
	\paragraph{Extension to quadratic programming (QP)} A potential future direction is to use the approach of the paper to establish error bounds for QP. The strategy consists of three steps: (1)  reducing the QP to an SOCP, (2) utilizing the error bound for the SOCP, and (3) translating the error bound to the QP setting. We leave the detail to future work.
	
	\section*{Acknowledgment} L. Ding and M. Udell were supported from NSF Awards IIS1943131 and CCF-1740822, the ONR Young Investigator
	Program, DARPA Award FA8750-17-2-0101, the Simons
	Institute, Canadian Institutes of Health Research, and Capital One.
	L. Ding would like to thank James Renegar for helpful discussions.
	
	\bibliographystyle{alpha}
	\bibliography{reference}
	
	\appendix
	\section{Proof for Section \ref{sec: fromerrorboundtosensitivityofsolution}}\label{sec: apprndixfromerrorboundtosensitivityofsolution}
	We first show the Slater's condition for \eqref{opt: primalconeperturbed} and its dual. Recall the Slater's condition
	for the original \eqref{opt: primalcone} means that the two points $x_0,y_0$ satisfying $(x_0,c-\Amap^*y)\in \intr(\cone) \times \intr(\cone^*)$.
	This implies that $\min\{\dist(x_0,\partial \cone),\dist(c-\mathcal{A}^*y_0,\partial \cone^*)\}\geq \eta$ for some constant $\eta>0$.
	We now construct $x_0'$ and $y_0'$ from $x_0$ and $y_0$.
	It can be easily verified if $\norm{\DeltaA} \leq \frac{\sigma_{\min>0}(\Amap)}{2}$,
	$ \frac{2}{\sigma_{\min>0}(\Amap)}(\norm{\Deltab}+\norm{\DeltaA}\norm{x_0})\leq \frac{\eta}{2}$,
	and $\norm{\Deltac}+\norm{\DeltaA}\norm{y_0}\leq \frac{\eta}{2}$, then the choice
	\begin{equation}
		\begin{aligned}
			x_0' = x_0 +\mathcal{A'}^\dagger (\Deltab -\DeltaA x_0), \quad y_0' = y_0
		\end{aligned}
	\end{equation}
	satisfy $\min\{\dist(x_0',\partial \cone),\dist(c'-(\mathcal{A}^*)'y_0', \partial (\cone^*))\}\geq \frac{\eta}{2}$, $\max\{
	\norm{x_0-x_0'},\norm{y_0-y_0'}\leq \frac{\eta}{2}$, and are feasible for\eqref{opt: primalconeperturbed} and its dual.
	
	Now for the boundedness condition of any solution $\xsol'$ to \eqref{opt: primalconeperturbed}. Using the previous constructed $x_0'$ and $y_0'$, we know
	\begin{equation}
		\begin{aligned}
			\inprod{\xsol'}{\cost'-(\Amap^*)'y_0'}& \leq \inprod{x_0'}{\cost'-(\Amap^*)'y_0'}
			=\inprod{c'}{x_0'} -\inprod{\bvector'}{y_0'}\\
			&=\inprod{\cost}{x_0}-\inprod{\bvector}{y_0} + \inprod{\Deltac}{x_0'}-\inprod{\Deltab}{y_0'}\\
			&\leq \inprod{\cost}{x_0}-\inprod{\bvector}{y_0} +\frac{\eta}{2}(\norm{x_0}+\frac{\eta}{2}+\frac{\sigma_{\min>0}(\Amap)}{2}\norm{y_0})
		\end{aligned}
	\end{equation}
	The rest is a simple consequence of the following lemma.
	\begin{lemma}
		Suppose $\cone$ is closed and convex. Given $\epsilon >0$ and $s_0\in (\cone^*)^\circ$ with
		$d = \dist(s_0,\partial \cone^*)>0$, there is a $B>0$ such that for any $x$ satisfying $x\in \cone$,
		and $\inprod{x}{s}\leq \epsilon$ for some $s$ with $\norm{s-s_0}\leq \frac{d}{2}$, its norm satisfies $\norm{x}\leq B$.
	\end{lemma}
	\begin{proof}
		Suppose such $B$ does not exist, then there is a sequence $(x_n,s_n)\in \cone \times \cone^*$ with
		$\norm{s_n-s_0}\leq \frac{d}{2}$, $\inprod{x_n}{s_n}\leq \epsilon$, and $\lim_{n\rightarrow \infty }\norm{x_n}=+\infty$.
		
		Now consider
		$(\frac{x_n}{\norm{x_n}},s_n)\in \cone \times \cone^*$. Since $\frac{x_n}{\norm{x_n}}$ and $s_n$ are bounded, we
		can choose a appropriate subsequence of $(\frac{x_n}{\norm{x_n}},s_n)$ which converges to certain $(x,s)\in \cone\times
		\intr(\cone^*)$
		as $\norm{s_n-s_0}\leq \frac{d}{2}$. Call the subsequence $(\frac{x_n}{\norm{x_n}},s_n)$  still. Using $\inprod{\frac{x_n}{\norm{x_n}}}{s_n}\leq \frac{\epsilon}{\norm{x_n}}$
		and $\norm{x_n}\rightarrow +\infty$, we see \[\inprod{x}{s}=0.\] This is not possible as $s\in \intr(\cone^*)$. Hence such $B$ must exist.
	\end{proof}
	\input DSCSC

\input lowerbound
	\input ConicDecomposition
	\input NumericalSimulation
\end{document}

%% file: defs.tex
\usepackage[usenames,dvipsnames]{xcolor}
\definecolor{purple}{rgb}{0.3,0.0,.4}

\usepackage{thmtools}
\usepackage{thm-restate}
\usepackage{hyperref}
\usepackage{xspace}
\usepackage{amssymb}
\usepackage{pifont}
\usepackage{accents}

\newcommand{\defn}{:\,=}
\newcommand{\BEAS}{\begin{eqnarray*}}
\newcommand{\EEAS}{\end{eqnarray*}}
\newcommand{\BEA}{\begin{eqnarray}}
\newcommand{\EEA}{\end{eqnarray}}
\newcommand{\BEQ}{\begin{equation}}
\newcommand{\EEQ}{\end{equation}}
\newcommand{\BIT}{\begin{itemize}}
\newcommand{\EIT}{\end{itemize}}
\newcommand{\BNUM}{\begin{enumerate}}
\newcommand{\ENUM}{\end{enumerate}}

\newcommand{\beas}{\begin{eqnarray*}}
\newcommand{\eeas}{\end{eqnarray*}}
\newcommand{\bea}{\begin{eqnarray}}
\newcommand{\eea}{\end{eqnarray}}
\newcommand{\beq}{\begin{equation}}
\newcommand{\eeq}{\end{equation}}
\newcommand{\bit}{\begin{itemize}}
\newcommand{\eit}{\end{itemize}}
\newcommand{\ben}{\begin{enumerate}}
\newcommand{\een}{\end{enumerate}}

\newcommand{\ba}{\begin{array}}
\newcommand{\ea}{\end{array}}
\newcommand{\bbm}{\begin{bmatrix}}
\newcommand{\ebm}{\end{bmatrix}}

\newcommand{\eg}{e.g., }
\newcommand{\ie}{i.e., }

\newcommand{\ones}{\mathbf 1}

\newcommand{\reals}{{\mbox{\bf R}}}

\newcommand{\sym}{{\mbox{\bf S}}}  

\newcommand{\Span}{\mathop{\bf span}}
\newcommand{\range}{\mathop{\bf range}}
\newcommand{\rank}{\mathop{\bf rank}}
\newcommand{\nullspace}{{\mathop {\bf nullspace}}}
\newcommand{\tr}{\mathop{\bf tr}}

\newcommand{\diag}{\mathop{\bf diag}}

\newcommand{\nnzero}{\mathbf{nnz}}

\newcommand{\twonorm}[1]{\left\|#1\right\|_2}
\newcommand{\norm}[1]{\left\|#1\right\|}
\newcommand{\fronorm}[1]{\left\|#1\right\|_{\mbox{\tiny{F}}}}
\newcommand{\opnorm}[1]{\left\|#1\right\|_{\mbox{\tiny{\textup{op}}}}}
\newcommand{\nucnorm}[1]{\left\|#1\right\|_*}

\newcommand{\dist}{\mathop{\bf dist{}}}

\newcommand{\argmin}{\mathop{\rm argmin}}

\newcommand{\intr}{\mathop{\bf int}}
\newcommand{\relint}{\mathop{\bf rel int}}

\theoremstyle{definition} 
\newtheorem{theorem}{Theorem}
\newtheorem{lemma}{Lemma}
\newtheorem{corollary}{Corollary}

\newtheorem{definition}{Definition}
\newtheorem{remark}{Remark}
\newtheorem{example}{Example}
\newcommand{\cone}{\mathcal{K}}
\newcommand{\Vspace}{\mathcal{V}}
\newcommand{\Amap}{\mathcal{A}}
\newcommand{\bvector}{b}
\newcommand{\cost}{c}
\newcommand{\DeltaA}{\Delta \Amap}
\newcommand{\Deltab}{\Delta b}
\newcommand{\Deltac}{\Delta c}
\newcommand{\inprod}[2]{\langle #1, #2 \rangle}
\newcommand{\xset}{\mathcal{X}_\star} 
\newcommand{\yset}{\mathcal{Y}_\star}
\newcommand{\pval}{p_\star}

\newcommand{\optgap}{\epsilon_{\mbox{\tiny opt}}}
\newcommand{\conefeas}[1]{#1_{-}}
\newcommand{\conpart}[1]{#1_{+}}
\newcommand{\bfx}{\mathbf{x}}

\newcommand{\onenorm}[1]{\|#1\|_1}
\newcommand{\rel}[1]{\mbox{rel}(#1)}
\newcommand{\affine}[1]{\mbox{affine}(#1)}
\newcommand{\SOC}[1]{\mbox{SOC}_{#1}}
\newcommand{\xsol}{x_\star} 
\newcommand{\ssol}{s_\star }
\newcommand{\ssoll}[1]{s_{\star,#1}}
\newcommand{\checkssol}{\check{s}_{\star}} 

\newcommand{\Xsol}{X_\star}
\newcommand{\Ssol}{S_\star}
\newcommand{\ysol}{y_\star }
\newcommand{\proj}[1]{\mathcal{P}_{#1}}

\newcommand{\dm}{n}
\newcommand{\cons}{m}

\newcommand{\hyperplane}{\mathcal{H}}
\newcommand{\Fface}{\mathcal{F}}

%% file: DSCSC.tex
\section{Equivalence between DSC and SC for LP and SDP}\label{sec: Equivalence between DSC and SC for LP and SDP}
We define the strict complementarity of LP and SDP, and 
show it is equivalent to DSC defined in Section \ref{sec: assumptions}. 
For a vector $x\in \reals^{\dm}$, denote $\nnzero(x)$ as its number of nonzeros. 
\begin{definition}\label{def: stricLPSDP}
	For LP, if there exists optimal primal dual pair $(\xsol,\ysol)\in \xset \times \yset \subset \reals_+^\dm \times \reals^\cons$ with 
	$\ssol = c-\Amap^* \ysol\in \reals_+^\dm$ such that 
	\[
	\nnzero(\xsol)+\nnzero(\ssol)=\dm,
	\] 
	we say \eqref{opt: primalcone} satisfies strict complementarity. Similary, for SDP, 
	if  there exists optimal primal dual pair $(\Xsol,\ysol)\in \xset \times \yset \subset \sym_+^\dm \times \reals^\cons$ with 
	$\Ssol = C-\Amap^* \ysol\in \sym_+^\dm$ such that 
	\[
	\rank(\Xsol)+\rank(\Ssol)=\dm,
	\] 
	we say \eqref{opt: primalcone} satisfies strict complementarity.
\end{definition} 

\begin{lemma}\label{lem: equivalenceSCandDSC}
	For both LP and SDP, under strong duality, the strict complementarity defined is equivalent to dual strict complementarity.
\end{lemma}
\begin{proof}
	For LP, under strong duality (see \eqref{eqn: strongduality}), we have for any optimal $\xsol$, 
	$\xsol\in \Fface_{\ssol}=\{x \in \reals_+^\dm \mid x_i = 0,\text{ for all } (\ssol)_i>0\}$. The relative 
	interior of $\Fface_{\ssol}$ is 
	\[
	\rel{\Fface_{\ssol}}= \{x \in \reals_+^\dm \mid x_i = 0\;\text{ for all } (\ssol)_i>0,\; \text{and}\; x_i >0 \;\text{ for all } (\ssol)_i=0\}.
	\]
	The equivalence between DSC and SC is then immediate.
	
	For SDP, under strong duality (see \eqref{eqn: strongduality}), we know that for any optimal $\Xsol$, 
	$\Xsol\in \Fface_{S_\star} = \{X\mid \inprod{X}{S_\star}=0, X\succeq 0\}=\{X\mid \range(X)\subset \nullspace(\Ssol),\; \text{and}\; X\succeq 0\}$. 
	The relative interior of $\Fface_{S_\star}$ is 
\[
\rel {\Fface_{S_\star}} = \{X\mid \range(X)=\nullspace(\Ssol),\; \text{and}\; X\succeq 0\}.
\]
The equivalence is immediate by using Rank-nullity theorem for $\Ssol$. 
\end{proof}

%% file: lowerbound.tex
\section{A lower bound on distance to optimality: $\twonorm{\proj{\Vspace_{\ssol}^\perp}(x)} \leq \dist (x,\xset)$ }\label{sec: lowerbound}

We have shown how to establish upper bounds on $\dist (x,\xset)$
with an upper bound on $\twonorm{\proj{\Vspace{\ssol}^\perp}(\conpart{x})}$.
In this section, we show that the same quantity $\twonorm{\proj{\Vspace{\ssol}^\perp}(\conpart{x})}$
also yields a lower bound. Hence, it is important to understand the behavior of $\twonorm{\proj{\Vspace{\ssol}^\perp}(\conpart{x})}$.
For simplicity, we suppose in this section that $x$ is feasible for the problem \eqref{opt: primalcone}
: in this case, $\twonorm{\proj{\Vspace{\ssol}^\perp}(\conpart{x})} = \twonorm{\proj{\Vspace{\ssol}^\perp}(x)}$.
The argument for infeasible $x$ is essentially the same.

First recall for feasible $x$, the only nonzero error metric is the suboptimality $\optgap(x)=\inprod{c}{x}-\pval$.
Also note that $\optgap(x)=0 \iff \twonorm{\proj{\Vspace_{\ssol}^\perp}(x)}$ using complementary slackness.
Hence, there is some nonnegative function $g:\reals \rightarrow \reals_+$ such that $g(\optgap(x))\leq  \twonorm{\proj{\Vspace_{\ssol}^\perp}(x)}$.
Now note that for any feasible $x$, we always have the lower bound on distance
 given by $\twonorm{\proj{\Vspace_{\ssol}^\perp}(x)}$ as $\Vspace_{\ssol}\supset \xset$,
 \begin{equation}\label{eq: lowerboundInequality}
g(\optgap(x))\leq \twonorm{\proj{\Vspace_{\ssol}^\perp}(x)} = \dist(x,\Vspace_{\ssol})\leq \dist (x,\xset).
 \end{equation}
The lower bound \eqref{eq: lowerboundInequality} hence shows that
$\twonorm{\proj{\Vspace_{\ssol}^\perp}(\conpart{x})}$
provides a lower bound on the distance $\dist (x,\xset)$,
and provides hope that this bound might scale with the suboptimality $\optgap(x)$.
We summarize our findings in the following theorem.

\begin{theorem}\label{thm: equivalence}
	Suppose there exists a increasing
	continuous $g$ with $g(0)=0$ so that
	for any $x$ feasible for Problem \eqref{opt: primalcone}\footnote{
	The assumption $x$ being feasible is just for convenience of presentation. The equivalence still holds for all $x$ with suboptimality, infeasibility, and conic infeasibility bounded above by some constant $\bar c>0$.}
	with $\epsilon(x)\leq \bar c$,
\begin{equation}
\begin{aligned}\label{eqn: equivalenceFirstterm}
\twonorm{\proj{\Vspace_{\ssol}^\perp}(x)} \geq g(\optgap(x)).
\end{aligned}
\end{equation}
Then the following inequality holds:
		\begin{equation}
	\begin{aligned}\label{eqn: equivalencewhole}
	\dist(x,\xset) \geq g(\optgap(x)).
	\end{aligned}
	\end{equation}
\end{theorem}

\begin{remark}
	We also have a partial converse for the above theorem that follows from the same proof.
	Assume $\dist(\proj{\Vspace_{\ssol}}(x),\Fface_{\ssol})$ is $0$.
	If the relation
	\eqref{eqn: equivalencewhole}
	holds for all feasible $x$ with $\optgap(x)\leq \bar c$, then
	$\twonorm{\proj{\Vspace_{\ssol}^\perp}(x)} \geq \frac{g(\optgap(x))}{1+\gamma \sigma_{\max}(\Amap)} $.
	Here $\gamma$ is defined in \eqref{eqn: erbstrictcomplslack}.
\end{remark}
\begin{proof}
	We have proved that \eqref{eqn: equivalenceFirstterm} implies \eqref{eqn: equivalencewhole}
	using the motivating logic laid out at the beginning of this section.
Conversely, assume the term $\dist(\proj{\Vspace_{\ssol}}(\conpart{x}),\Fface_{\ssol})$ is zero, $x$ is feasible,
$\optgap(x)\leq \bar{c}$,
and the inequality \eqref{eqn: equivalencewhole} holds.
We see that for all feasible $x$,
\begin{equation}
\begin{aligned}
\beta g(\optgap(x)) \leq \dist(x,\xset)
\overset{(a)}{\leq} (1+\gamma\sigma_{\max}(\Amap))\twonorm{\proj{\Vspace_{\ssol}^\perp}(x)}.
\end{aligned}
\end{equation}
where we use \eqref{eqn: frameworkSimplerSituation} for step $(a)$.
\end{proof}

%% file: ConicDecomposition.tex
\section{Conic Decomposition}\label{sec: Conic Decomposition}
In the main paper, we decompose a general $x\in \EE$ according to 
the subspace $ \Vspace_{\ssol}$. 
A different decomposition uses the cone $\Fface_{\ssol}$: 
every $x\in \EE$ admits the conic decomposition 
$x = \proj{\Fface_{\ssol}}(x) + \proj{\Fface_{\ssol}^{\circ}}(x)$ where $\Fface_{\ssol}^{\circ}$ 
is the \emph{polar cone} of $\Fface_{\ssol}$, i.e., the negative dual cone $-\Fface^{*}_{\ssol}$.

\begin{theorem}\label{thm: basicframeworkConic}
	Suppose strong duality and dual strict complementarity hold.
	Then for some
	constants $\gamma,\gamma'$ described in Lemma \ref{lem: linearegularitystrictcomplementarity} and for all $x\in \EE$,
	we have
	\begin{equation}
		\begin{aligned}\label{eqn: ConicDecompositionFramework}
			\dist(x,\xset) \leq & (1+\gamma\sigma_{\max}(\Amap))\twonorm{\proj{\Fface_{\ssol}^\circ}(x)}+\gamma \twonorm{\Amap(x)-b}.
		\end{aligned}
	\end{equation}
\end{theorem}
Let us compare the above bound \eqref{eqn: ConicDecompositionFramework} and \eqref{eqn: framework} in Theorem  \ref{thm: basicframework}. 
To make the comparison easier, first note that from the proof of Theorem \ref{thm: basicframework}, we can bound the distance from $x$ to $\xset$ 
using the decomposition $x = \proj{\Vspace_{\ssol}}(x) +\proj{\Vspace_{\ssol}^\perp}(x) $: 
\begin{equation}\label{eq: intermediateFramework}
	\begin{aligned} 
		\dist(x,\xset) \leq  (1+\gamma\sigma_{\max}(\Amap))\twonorm{\proj{\Vspace_{\ssol}^\perp}(x)}+\gamma \twonorm{\Amap(x)-b}
+	\gamma'\dist(\proj{\Vspace_{\ssol}}(x),\Fface_{\ssol}).
	\end{aligned} 
\end{equation}
The bound \eqref{eqn: framework} in Theorem  \ref{thm: basicframework} is further obtained via the decomposition  $x=\conpart{x}+\conefeas{x}$.

Comparing \eqref{eq: intermediateFramework} and \eqref{eqn: ConicDecompositionFramework}, we find that 
there is an extra term $\gamma'\dist(\proj{\Vspace_{\ssol}}(x),\Fface_{\ssol})$ in \eqref{eq: intermediateFramework} and 
the term $\twonorm{\proj{\Vspace_{\ssol}^\perp}(x)}$ in \eqref{eq: intermediateFramework} is replaced by $\twonorm{\proj{\Fface_{\ssol}^\circ}(x)}$.
Since $\twonorm{\proj{\Fface_{\ssol}^\circ}(x)} \geq \twonorm{\proj{\Vspace_{\ssol}^\perp}(x)}$, it is not immediately clear which bound 
is tighter. 

A more subtle difference between  \eqref{eqn: ConicDecompositionFramework} and \eqref{eqn: framework} in Theorem  \ref{thm: basicframework} 
is that we are not able to further bound $\twonorm{\proj{\Fface_{\ssol}^\circ}(x)}$ using the decomposition $x=\conpart{x}+\conefeas{x}$
with respect to $\mathcal K$.
We reach this impasse because the projection operator $\proj{\Fface_{\ssol}^\circ}$ 
is not linear and so we cannot rely on the triangle inequality $\twonorm{\proj{\Fface_{\ssol}^\circ}(x)} \leq \twonorm{\proj{\Fface_{\ssol}^\circ}(\conpart{x})} + \twonorm{\proj{\Fface_{\ssol}^\circ}(\conefeas{x})}$.
Hence we cannot bound $\twonorm{\proj{\Fface_{\ssol}^\circ}(x)}$  using conic infeasibility. 

Thus, to use \eqref{eq: intermediateFramework}, we must use $x$ (which may be infeasible with respect to the cone $\mathcal K$) directly
to bound $\twonorm{\proj{\Fface_{\ssol}^\circ}(x)}$. 
We now consider how to bound this term for the cases considered in the main text. 

\paragraph{Case $\ssol =0$.} For $\ssol =0$, we have $\Fface_{\ssol}= \cone$. Thus $\Fface_{\ssol}^\circ = \cone^\circ$ and $\proj{\Fface_{\ssol}^\circ}(x) = x_-$. 
For $\ssol \in \intr(\cone^*)$, we have $\Fface_{\ssol}=\{0\}$. Thus  $\Fface_{\ssol}^\circ = \EE$ and 
$\twonorm{\proj{\Fface_{\ssol}^\circ}(x)} =\twonorm{x}\leq \twonorm{x_+} + \twonorm{x_-}$. 
We can bound  $\twonorm{x_+} $ as in Lemma \ref{lem: boundPVperpx+}.

\paragraph{Case $\ssol =\reals_+$.} For $\cone = \reals_+$, the projection of $\proj{\Fface_{\ssol}^\circ}(x) = x - (x_{I_{\ssol}^c})_+= 
(x_+)_{I_{\ssol}} + x_-$ where 
$x_{I_{\ssol}^c}$ is the vector $x$ zeroing all entries in the support of $\ssol$ and $(x_+)_{I_{\ssol}} $ is 
the vector $x_+$ zeroing out all entries not in the support of $\ssol$. Hence, 
$\twonorm{\proj{\Fface_{\ssol}^\circ}(x) }\leq \twonorm{(x_+)_{I_{\ssol}}} + \twonorm{x_-}$ and we can 
further bound $\twonorm{(x_+)_{I_{\ssol}}}$ as in Lemma \ref{lem: boundPVperpx+}. 

\paragraph{Case $\ssol =\SOC{n}$.} 
For $\cone = \SOC{n}$, considering the nontrivial case $\ssol \not = 0$, 
the projection of $\proj{\Fface_{\ssol}^\circ}(x)$ is 
\[
 x -(\inprod{x}{\checkssol})_+ \checkssol = x_+ -(\inprod{x}{\checkssol})_+ \checkssol +x_-. 
\]
Thus we have 
\begin{equation}
	\begin{aligned} 
\twonorm{\proj{\Fface_{\ssol}^\circ}(x) - \proj{\Vspace^\perp_{\ssol}}(x_+)}
  & = \twonorm{\left( \inprod{x_+}{\checkssol} - (\inprod{x}{\checkssol})_+ \right)\checkssol +x_-} \\ 
  & \leq  |\inprod{x_+}{\checkssol}- (\inprod{x}{\checkssol})_+ | + \twonorm{x_-} \\
  & \overset{(a)}{\leq}  |\inprod{x_+}{\checkssol}| + \twonorm{x_-} .
\end{aligned} 
\end{equation}
In the step $(a)$, we use $x = x_+ + x_-$ and $\inprod{x_-}{\checkssol}\leq 0$. Hence, one 
can bound $\twonorm{ \proj{\Fface_{\ssol}^\circ}(x) }$ by combining the above bound and 
the bound on $\proj{\Vspace^\perp}(x_+)$ established in the main text. 

\paragraph{Case $\ssol = \sym_+^n$.} 
For $\cone = \sym_+^{n}$, $\proj{\Fface_{\Ssol}^\circ}(X)$ is 
\[
\proj{\Fface_{\Ssol}^\circ}(X) = X - V(V^\top XV)_+V^\top= X_+ - V(V^\top XV)_+V^\top +X_-.
\]
Since $\proj{\Vspace_{\Ssol}^\top }(X+) = X_+ -  VV^\top X_+VV^\top$, we know 
\begin{equation}
	\begin{aligned} 
		\fronorm{\proj{\Fface_{\ssol}^\circ}(X) - \proj{\Vspace^\perp}(X_+)}
		& = \fronorm{ V(V^\top XV)_+V^\top-  VV^\top X_+VV^\top +X_-}\\ 
		& \overset{(a)}{\leq}  \fronorm{(V^\top XV)_+-  V^\top X_+V} +\fronorm{X_-}\\
		& \overset{(b)}{\leq }2\fronorm{X_-}.
	\end{aligned} 
\end{equation}
In step $(a)$, we use the fact that $V$ has orthonormal columns, and 
in step $(b)$, we use the fact that $V^\top X_+V$ is still positive semidefinite 
and projection to the convex set $\sym_+^r$ is nonexpansive. 
Thus, we can bound
$\fronorm{\proj{\Fface_{\ssol}^\circ}(X)}$ using the result for 
$\fronorm{ \proj{\Vspace^\perp}(X_+)}$ in the main text. 

\paragraph{Case $x$ is feasible.} 
Finally, note that when $x$ is feasible for \eqref{opt: primalcone}, 
then in each of the five cases considered in the paper, 
the bound \eqref{eqn: ConicDecompositionFramework} and 
the bound \eqref{eqn: framework} in Theorem \ref{thm: basicframework} coincide. 
\begin{proof}
	Recall Lemma \ref{lem: linearegularitystrictcomplementarity} establishes an error bound
	only for $x\in \Vspace_{\ssol}$.
	Using the conic decomposition into the face $\Fface$ and its polar, 
	for any $x\in \EE$,
	\[
	x = \proj{\Fface_{\ssol}}(x) + \proj{\Fface_{\ssol}^{\circ}}(x).
	\]
	This decomposition immediately gives
	\begin{equation}
		\begin{aligned}\label{eqn: conicDecomp}
			\dist(x,\xset) & \leq \twonorm{\proj{\Fface_{\ssol}^\circ}(x)} +
			\dist(\proj{\Fface_{\ssol}}(x), \xset).
		\end{aligned}
	\end{equation}
	The second term $\dist(\proj{\Fface_{\ssol}}(x), \xset)$ can be bounded using Lemma \ref{lem: linearegularitystrictcomplementarity} as $\Fface_{\ssol} \subset \Vspace_{\ssol}$:
	\begin{equation}\label{eqn: ConicDecompandpreviouscombine}
		\dist(\proj{\Fface_{\ssol}}(x), \xset)\leq \gamma\twonorm{\Amap(\proj{\Fface_{\ssol}}(x))-b}.
	\end{equation}
Using the decomposition $x = \proj{\Fface_{\ssol}}(x) + \proj{\Fface_{\ssol}^{\circ}}(x)$ again for the term $\twonorm{\Amap(\proj{\Fface_{\ssol}}(x))-b}$,
we reach the bound \eqref{eqn: ConicDecompositionFramework}.
\end{proof}

%% file: NumericalSimulation.tex
\section{Numerical simulation for the bound \eqref{eqn: errorboundLPSOCPSDPandOthers}}\label{sec: Conic Decomposition}

Here we numerically verify the 
correctness of the inequality \eqref{eqn: errorboundLPSOCPSDPandOthers} for feasible $x$:
\begin{equation}
	\begin{aligned}\label{eqn: errorboundLPSOCPSDPandOthersFeasibleNorm}
	\dist(x,\xset) & \leq (1+\gamma\sigma_{\max}(\Amap))f(\optgap(x),\norm{x}).
	\end{aligned}
\end{equation}
The function $f$ can be found in Table \ref{tb: PAndconstantAndBoundf}.

\paragraph{Experiment setup} We generated a random instance $\mathcal{A},b,c$ for each of LP, SOCP, and SDP. We solved the corresponding conic problem and obtain the optimal solution $\xsol$ and a dual optimal $\ssol$. We numerically verified  that the strict complementarity (by checking Definition \ref{def: stricLPSDP} for LP and SDP and \eqref{eqn: strict complementarityset} for SOCP)  and the uniqueness of the primal (by checking whether $\sigma_{\min}(\Amap_{\Vspace_{\star}})>0$) both hold for the three cases.  
We compute $\gamma = \frac{1}{\sigma_{\min}(\Amap_{\Vspace_{\star}})}$ according to Lemma 2. Next, we randomly perturbed the solution $\xsol$ $70$ many times and obtained possibly infeasible $x_i', i = 1,\dots, 70$. We then projected $x_i'$ to the feasible set to obtain $x_i$. Finally, we plotted the suboptimality of $x_i$ versus the distance to $\xsol$ (in blue), and the the suboptimality of $x_i$ versus the bound $(1+\gamma\sigma_{\max}(\Amap))f(\optgap(x_i),\norm{x_i})$ (in red) in Figure \ref{fig:three graphs opt distance bound}.

From Figure \ref{fig:three graphs opt distance bound}, we observe that the bounds are above the actual distance which assures \eqref{eqn: errorboundLPSOCPSDPandOthersFeasibleNorm}. However, for SDP, the bounds appear to be looser compared to the cases of LP and SOCP. 

\begin{figure}
     \centering
     \begin{subfigure}[b]{0.48\textwidth}
         \centering
         \includegraphics[width=\textwidth]{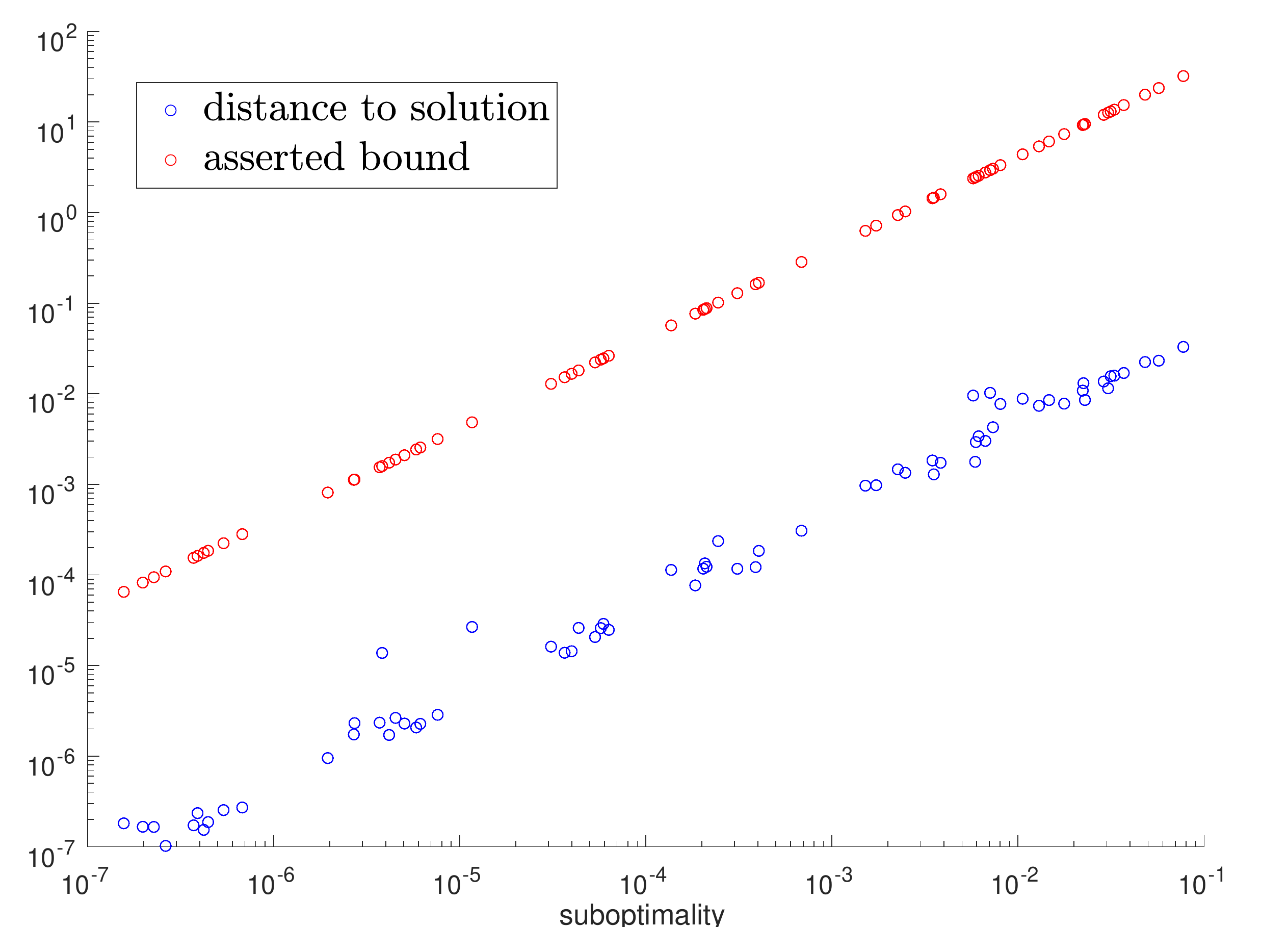}
         \caption{LP}
         \label{fig: LP}
     \end{subfigure}
     \hfill
     \begin{subfigure}[b]{0.48\textwidth}
         \centering
         \includegraphics[width=\textwidth]{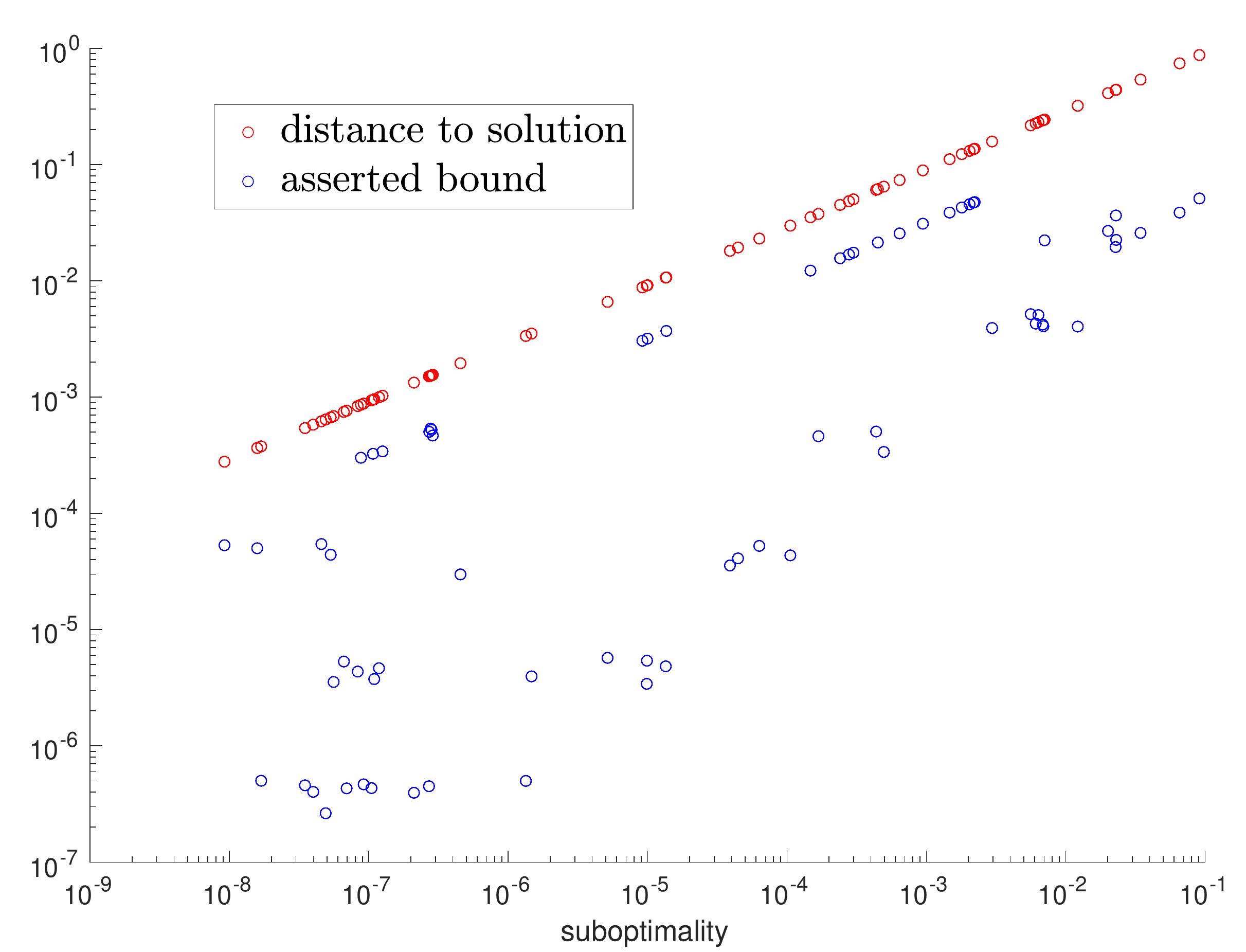}
         \caption{SOCP}
         \label{fig: SOCP}
     \end{subfigure}
     \hfill
     \begin{subfigure}[b]{0.48\textwidth}
         \centering
         \includegraphics[width=\textwidth]{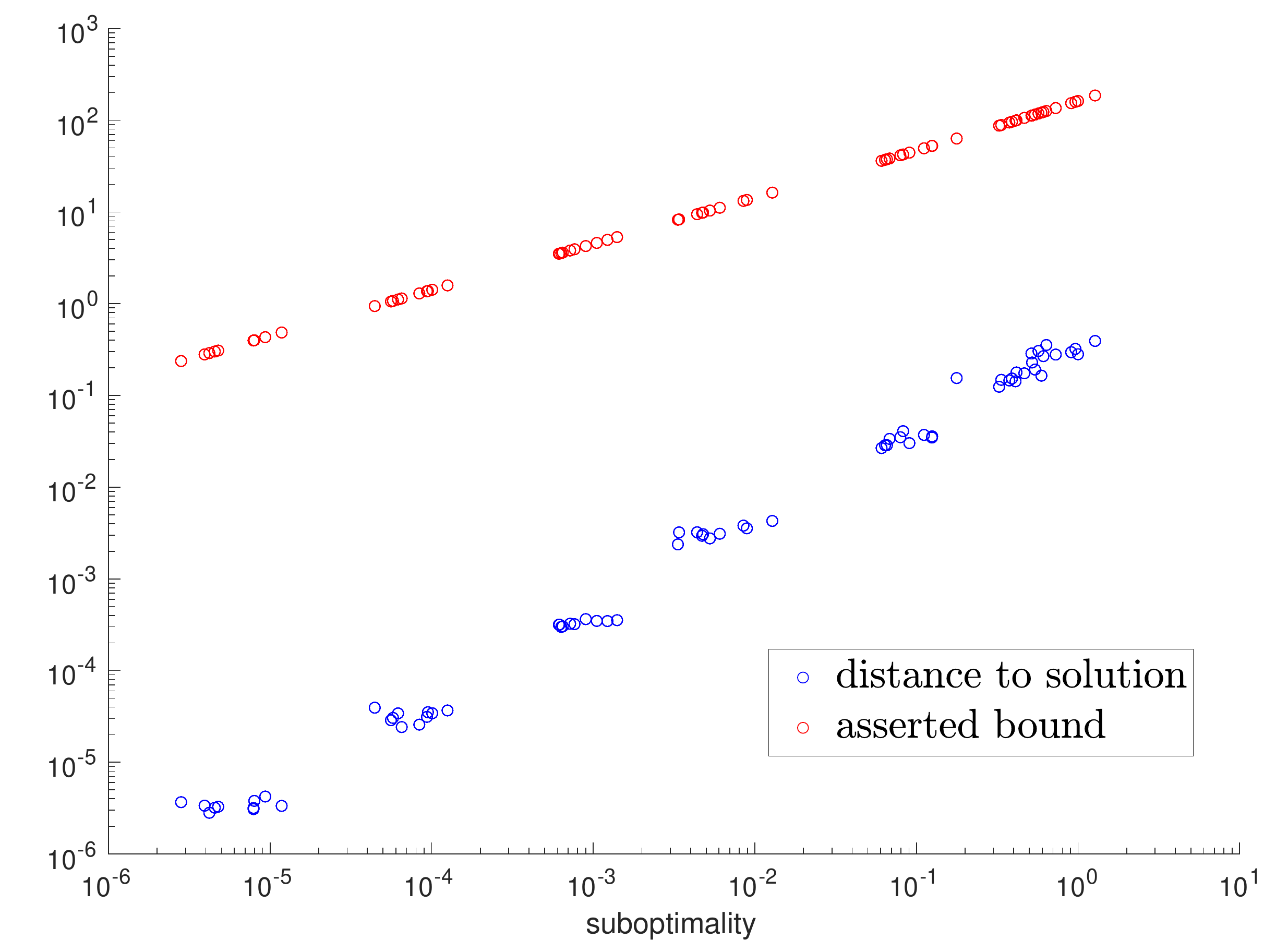}
         \caption{SDP}
         \label{fig: SDP}
     \end{subfigure}
        \caption{Verification of the inequality \eqref{eqn: errorboundLPSOCPSDPandOthersFeasibleNorm}. The asserted bound in red is $(1+\gamma\sigma_{\max}(\Amap))f(\optgap(x),\norm{x})$. The blue points represent suboptimality versus distance to solution.}
        \label{fig:three graphs opt distance bound}
\end{figure}